\documentclass[preprint,11pt]{elsarticle}
\usepackage{amssymb}
\usepackage{amsthm}
\usepackage{amsmath}
\usepackage{amsfonts}
\usepackage{overpic}
\usepackage{setspace}
\usepackage{subfigure}
\usepackage{hyperref}
\usepackage{amsmath}
\usepackage{amssymb}
\usepackage{mathtools}
\usepackage{listings}
\usepackage{xcolor}
\usepackage{hyperref}

\lstdefinelanguage{Mathematica}{
  morekeywords={Solve, D, FullSimplify, Sum, Subscript, CoefficientList, 
                Variables, Expand, Normal, Series, Factor, Union},
  sensitive=true,
  morecomment=[l]{(*},
  morecomment=[l]{*)},
  morestring=[b]",
}

\usepackage[lmargin=2.3cm,tmargin=2.3cm,bmargin=2.3cm,rmargin=2.3cm]{geometry}
\usepackage{indentfirst}

\makeatletter
\hypersetup{colorlinks=true}
\AtBeginDocument{\@ifpackageloaded{hyperref}
  {\def\@linkcolor{blue}
    \def\@citecolor{red}
     \def\@urlcolor{orange}
   
}{}}
\makeatother

\usepackage{tikz, tkz-base, tkz-euclide}

\newcommand{\s}{\Sigma}

\newcommand{\e}{\varepsilon}

\newcommand{\Cr}{\mathcal{C}^{\infty}}
\newcommand{\Xr}{\chi^{\infty}}
\newcommand{\Or}{\Omega^{\infty}}
\newcommand{\rn}[1]{\mathbb{R}^{#1}}

\newcommand{\er}{\mathcal{O}}

\newcommand{\ag}{\alpha}
\newcommand{\bg}{\beta}
\newcommand{\cg}{\gamma}
\newcommand{\dg}{\delta}
\newcommand{\sgn}{\textrm{sgn}}

\newtheorem {theorem} {Theorem}
\newtheorem {prop} {Proposition}

\newtheorem {lemma}  {Lemma}

\newtheorem {remark}{Remark}

 \newtheorem{definition}{Definition}

\definecolor{verde}{rgb}{0.0,0.5,0.0}
\definecolor{azul}{rgb}{0,0,128}
\definecolor{roxo}{rgb}{0.44,0.16,0.39}
\definecolor{vinho}{rgb}{0.5,0.0,0.13}
\definecolor{lilas1}{rgb}{0.6,0.33,0.73}
\definecolor{rosa}{rgb}{0.84,0.04,0.33}
\definecolor{mostarda}{rgb}{0.91,0.41,0.17}
\definecolor{mostarda2}{rgb}{1.0,0.66,0.07}
\journal{Journal of...}

\begin{document}
\onehalfspacing 

\begin{frontmatter}
	
	\title{Bifurcation Analysis of 3D Filippov Systems around Cusp-Fold Singularities}
	
			\author{Oscar A. R. Cespedes}
	\ead{osaramirezc@udistrital.edu.co}	
	\address{ Facultad Ciencias Matemáticas y Naturales,\\  Universidad Distrital Francisco José de Caldas\\ Bogotá,  Colombia.}

		\author{Rony Cristiano}
	\ead{rony.cristiano@ufg.br}	
	\address{Department of Mathematics, UFG, IME\\ Goi\^ania-GO, 74690-900, Brazil.}
	
	\author{Ot\'avio M. L. Gomide}
	\ead{otavio.marcal@ufg.br}	
	\address{Department of Mathematics, UFG, IME\\ Goi\^ania-GO, 74690-900, Brazil.}

	\begin{abstract}

		This paper investigates the local behavior of 3D Filippov systems $Z=(X,Y)$, focusing on the dynamics around cusp-fold singularities. These singular points, characterized by cubic contact of vector field $X$ and quadratic contact of vector field $Y$ on the switching manifold, are structurally unstable under small perturbations of $Z$, giving rise to significant bifurcation phenomena.
		
		We analyze the bifurcations of a 3D Filippov system around an invisible cusp-fold singularity, providing a detailed characterization of its crossing dynamics under certain conditions. We classify the characteristics of the singularity when it emerges generically in one-parameter families (a codimension-one phenomenon), and we show that no crossing limit cycles (CLCs) locally bifurcate from it in this particular scenario. When the vector fields $X$ and $Y$ are anti-collinear at the cusp-fold singularity, we provide conditions for the generic emergence of this point in two-parameter families (a codimension-two phenomenon). In this case, we show that the unfolding of such a singularity leads to a bifurcating CLC, which degenerates into a fold-regular polycycle (self-connection at a fold-regular singularity).
		
		Furthermore, we numerically derive the polycycle bifurcation curve and complete the two-parameter bifurcation set for a boost converter system previously studied in the literature. This allows the identification of parameter regions where the boost converter system exhibits a CLC in its phase portrait,  providing a  understanding of its complex dynamics.

	\end{abstract}
	
	\begin{keyword}
		\texttt Filippov systems\sep Bifurcation Theory \sep cusp-fold singularity \sep boost converter\sep crossing limit cycle \sep polycycle\\

	\end{keyword}
	
\end{frontmatter}

\tableofcontents

\section{Introduction}

 Piecewise smooth dynamical systems (PSDS for short) are often used to model physical, biological, social, and other phenomena that present some type of discontinuity in their motion (see \cite{DiB}) and for this reason, such research topic has motivated many works over the last few years. The theoretical understanding of three-dimensional PSDS is not yet well developed due to the wealth of phenomena that arise in this scenario and their complexity. Although recent work in the area has shown that we have a reasonable understanding of planar PSDS, there are still many open questions.  

The Filippov's convention (see \cite{F} for more details) is one of the most useful approaches to define the notion of solutions of a PSDS with a regular {\it switching manifold} $\Sigma$, and that's why it has been widely used in real applications (see \cite{ RC2016, RC2021, Jong2004, RODRIGUES2020}). In this context, some special singularities lying on $\Sigma$, known as $\Sigma$-singularities, play an important role in the comprehension of the dynamics of a PSDS and their characterization is the starting point to construct a solid theory.

$\Sigma$-singularities are well understood in planar PSDS systems, see, for example, \cite{BLS, GTS, KRG, ST,T1, T6,T2, T3}. But when it comes to three-dimensional PSDS systems, the behavior around a $\s$-singularity is extremely more complicated and still little explored. The local structure around some $3D$ $\s$-singularities is studied in \cite{CJ1,CJ2,CJn, RC2018}, where the characterization of the dynamics around a generic $3D$ $\s$-singularity, the so-called $T$-singularity, is highlighted. The local behavior around a $T$-singularity was one of the principal missing points to characterize the codimension $0$ singularities in dimension $3$. Recently, this problem was solved in  \cite{GT} and the description of the behavior around a $T$-singularity as well as the characterization of the locally structurally stable systems were established. In general, the T-singularity in a $3D$ Filippov system $Z=(X,Y)$ is a point $p\in\s$ that appears at the transversal intersection of the sets $S_X$ and $S_Y$ of points where both vector fields $X$ and $Y$ has an invisible quadratic (invisible fold) contact with $\s$, respectively. Another interesting $3D$ $\s$-singularity is the cusp-fold singularity, which occurs when $X$ has quadratic contact and $Y$ has cubic contact with $\s$ at $p$, or vice versa. Unlike the $T$-singularity, a Filippov system is not structurally stable around a cusp-fold singularity since small perturbations turns this double-tangency point into a two-fold singularity.
In particular, when $X$ has an invisible quadratic fold, the cusp-fold singularity $p$ may unfold into a $T$-singularity. In \cite{JTT13}, the authors have developed a topological normal form around a cusp-fold singularity, and the basin of attraction around such point has been studied for some models in  \cite{CT14, CTT16}. Recently, in \cite{ER24}, the author have studied the limit sets of a class of 3D Filippov systems near a cusp-fold singularity. As far as we know, there is no general description of the dynamics of a Filippov system around a cusp-fold singularity.

In light of this, our work is mainly devoted to investigate the qualitative behavior around a cusp-fold singularity and characterize the unfoldings of a codimension-two bifurcation that occurs when the vector fields, on both sides of the switching manifold, are anticollinear at this singularity. Although the topic of bifurcations at double tangential singularities has appeared in several works in the area, in general they focus on the study of the codimension-one bifurcation that occurs at the T-singularity and gives rise to a crossing limit cycle (CLC for short), see \cite{CJ2,RC2018}. A few works address the codimension-two bifurcation at the T-singularity, such as the Pseudo-Bautin bifurcation studied in \cite{Castillo}. 

It is worth mentioning that, generically, when a vector field $X$ has a cubic contact with $\s$ at $p$, then $p$ is contained in a local curve of points where $X$ has a quadratic contact with $\s$. It means that, when a cusp-fold singularity is unfolded then curves of visible and invisible fold-regular points of $Z=(X,Y)$ appear, which occur when the contact of $X$ is quadratic and $Y$ is transverse to $\s$. In this scenario, such visible fold-regular singularities may give rise to self-connections bifurcating from a cusp-fold singularity. Global connections based on fold-regular singularities (referred as fold-regular polycycles) have been extensively studied in planar Filippov systems (see \cite{FPT15, GTS}, for instance), and its characterization for generic one-parameter families of 3D-Filippov systems has been completed described in \cite{GT22}. More specifically, they prove that, such global connection unfolds into either a CLC or a closed trajectory containing a sliding segment (sliding cycle).


Based on the study developed in this work we show that, under certain conditions on the parameters of a general 3D-Filippov system, either a CLC or a fold-regular polycycle can bifurcate from the cusp-fold singularity. Such a singularity is not structurally stable and its appearance in the phase portrait of the system is related to a bifurcation point represented in a two-parameter bifurcation set. From this bifurcation point, two branches of bifurcation curves emanate, one referring to the local bifurcation at the T-singularity that gives rise to a CLC, and the other referring to the global bifurcation where this CLC becomes a fold-regular polycycle.  In this context, we provide explicit conditions on the parameters of a general 3D-Filippov system for which a CLC becomes a fold-regular polycycle.


The main results and original contributions obtained in the present work are:
\begin{itemize}
\item[(i)] Description of conditions under which no unfolding of a Filippov system $Z_0$ with a cusp-fold singularity admits CLCs bifurcating from this singularity, see Theorem \ref{thmnonexistence}. 
\item[(ii)] Explicit conditions on a Filippov system $Z_0$ with a cusp-fold singularity such that a $k$-dimensional unfolding of $Z_0$ exhibits either CLCs or fold-regular polycycles bifurcating from the singularity,  see Theorem \ref{Teo-CLC-Policiclo}.
\item[(iii)] A complete description of the unfolding dynamics of the semi-linear part of the normal form around a codimension-two cusp-fold singularity is provided, see Subsection \ref{Toy-Model}.
\item[(iv)] In \cite{TSboost} the authors investigated the occurrence of a bifurcation at the T-singularity and the existence of CLCs in the model of a boost power converter under a sliding mode control strategy. Based on numerical results they showed that a CLC of this system can become a polycycle by varying the converter load parameter. However, such a result is given only in terms of the variation of a single parameter, the converter load parameter, see Figures 11 and 12 in \cite{TSboost}. In addition, the two-parameter bifurcation set obtained does not contain the polycycle bifurcation curve (see Figure 7 in \cite{TSboost}), therefore, from the presented result it is not possible to know the region in the parameter plane for which the system has a CLC in its phase portrait, being possible only to identify for which values of the parameters it bifurcates from the T-singularity. Based on the above, in this present work we numerically obtain the polycycle bifurcation curve and thus complete the two-parameter bifurcation set of the boost converter system, see Subsection \ref{Sec-Boost} and Figure \ref{Fig-Set-Bif}. From this result it is possible to know the region in the parameter plane for which the boost converter system has a CLC in its phase portrait.
\end{itemize}

\section{Preliminaries}

This section is devoted to present some basic concepts of Filippov systems in order to clarify the notations contained in this paper. Moreover, we enunciate some key results that are needed to prove our main results.

Let \( M \) be an open, bounded, connected subset of \( \mathbb{R}^3 \), and let \( f: M \rightarrow \mathbb{R} \) be a smooth function with \( 0 \) as a regular value. Consequently, \( \Sigma = f^{-1}(0) \) is an embedded codimension-one submanifold of \( M \), spliting it into the regions \( M^{\pm} = \{ p \in M \mid \pm f(p) > 0 \} \).\\

A \textbf{piecewise smooth vector field} defined on \( M \) with switching manifold \( \Sigma \) is expressed as  
\begin{equation}\label{filippov_rewritten}  
Z(p) = F_1(p) + \text{sgn}(f(p)) F_2(p),  
\end{equation}  
where \( F_1(p) = \frac{X(p) + Y(p)}{2} \) and \( F_2(p) = \frac{X(p) - Y(p)}{2} \), with \( X, Y \) being \( \mathcal{C}^\infty \) vector fields. We denote \( Z = (X, Y) \). The set  of  \( \mathcal{C}^\infty \) piecewise smooth vector fields is denoted by \( \Omega^\infty \).  

The \textbf{Lie derivative} \( Xf(p) \) of \( f \) along \( X \in \mathcal{X}^\infty \) at \( p \in \Sigma \) is defined as \( Xf(p) = \langle X(p), \nabla f(p) \rangle \). The \textbf{tangency set} between \( X \) and \( \Sigma \) is given by \( S_X = \{ p \in \Sigma \mid Xf(p) = 0 \} \).  

\begin{remark}
The Lie derivative is well-defined for a germ \( \widetilde{X} \in \mathcal{X}^\infty \), as all representatives agree on \( \Sigma \).  
\end{remark}

For \( X_1, \dots, X_k \in \mathcal{X}^\infty \), higher-order Lie derivatives are recursively defined as  
\[ X_k \cdots X_1 f(p) = X_k (X_{k-1} \cdots X_1 f)(p), \]  
i.e., \( X_k \cdots X_1 f(p) \) is the Lie derivative of \( X_{k-1} \cdots X_1 f \) along \( X_k \) at \( p \). When \( X_i = X \) for all \( i \), we write \( X^k f(p) \).  

For \( Z = (X, Y) \in \Omega^\infty \), the switching manifold \( \Sigma \) generically splits at the closure of three disjoint open regions:  
\begin{itemize}
    \item \textbf{Crossing region:} \( \Sigma^c = \{ p \in \Sigma \mid Xf(p) Yf(p) > 0 \} \).  
    \item \textbf{Sliding region:} \( \Sigma^s = \{ p \in \Sigma \mid Xf(p) < 0, Yf(p) > 0 \} \).  
    \item \textbf{Escaping region:} \( \Sigma^e = \{ p \in \Sigma \mid Xf(p) > 0, Yf(p) < 0 \} \).  
\end{itemize}
The \textbf{tangency set} of \( Z \) is \( S_Z = S_X \cup S_Y \). Thus, \( \Sigma \) partitions as \( \Sigma^c \cup \Sigma^s \cup \Sigma^e \cup S_Z \). See Figure \ref{dicfig_rewritten}.  

	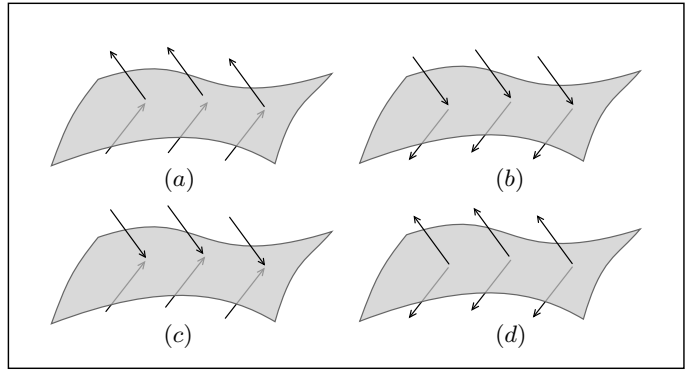
\begin{figure}[h!]
	\centering
	\bigskip
	\begin{overpic}[width=9cm]{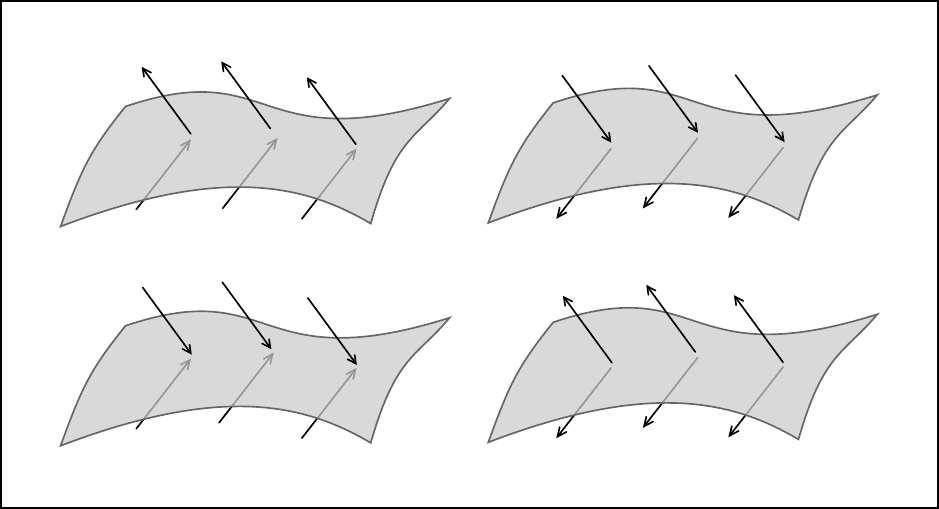}
			\put(23,27){{\footnotesize $(a)$}}		
			\put(72,27){{\footnotesize $(b)$}}		
			\put(23,4){{\footnotesize $(c)$}}		
			\put(72,4){{\footnotesize $(d)$}}									
		\end{overpic}
		\bigskip
		\caption{Regions in $\Sigma$: $\Sigma^{c}$ in $(a)$ and $(b)$, $\Sigma^{s}$ in $(c)$ and $\Sigma^{e}$ in $(d)$.   }	\label{dicfig_rewritten}
	\end{figure}

Following Filippov's convention, solutions of \( Z = (X, Y) \in \Omega^\infty \) are defined as follows. For \( p \in \Sigma^s \cup \Sigma^e \), the local solution is defined by the \textbf{sliding vector field}  
\begin{equation}  
F_Z(p) = \frac{1}{Yf(p) - Xf(p)} \left( Yf(p) X(p) - Xf(p) Y(p) \right).  
\end{equation}  
Here, \( F_Z \) is a \( \mathcal{C}^{r-1} \) vector field tangent to \( \Sigma^s \cup \Sigma^e \). Its critical points in these regions are called \textbf{pseudo-equilibria} of \( Z \).  

To study \( F_Z \), we introduce the \textbf{normalized sliding vector field}  
\begin{equation}  
F_Z^N(p) = Yf(p) X(p) - Xf(p) Y(p),  
\end{equation}  
defined for \( p \in \Sigma^s \). This extension simplifies analysis, as \( F_Z^N \) admits a \( \mathcal{C}^\infty \) continuation beyond \( \Sigma^s \). On a connected component \( R \) of \( \Sigma^s \), \( F_Z^N \) reparameterizes \( F_Z \), preserving phase portraits. For \( R \subset \Sigma^e \), \( F_Z^N \) is a negative reparameterization, reversing orbit orientations.  

If $p\in\s^c$, then the orbit of $Z=(X,Y)\in\Or$ at $p$ is defined as the concatenation of the orbits of $X$ and $Y$ at $p$. Nevertheless, if $p\in\s\setminus\s^c$, then it may occur a lack of uniqueness of solutions. In this case,  the flow of $Z$ is multivalued and any possible trajectory passing through $p$ originated by the orbits of $X$, $Y$ and $F_Z$ is considered as a solution of $Z$. More details can be found in \cite{F,GTS}.

In this scenario, we give special attention to the following kinds of points of $\s$ which display an interesting local behavior.

\begin{definition}\label{sigmasing}
	Let $Z=(X,Y)\in\Or$, a point $p\in \s$ is said to be:
	\begin{enumerate}[i)]
		\item a \textbf{tangential singularity} of $Z$ provided that $Xf(p)Yf(p)=0$ and $X(p), Y(p)\neq 0$;
		\item a \textbf{$\s$-singularity} of $Z$ provided that $p$ is either a tangential singularity, an equilibrium of $X$ or $Y$, or a pseudo-equilibrium of $Z$. 
	\end{enumerate}
\end{definition}


Since we have the presence of tangencies of a vector field with the boundary of a manifold, it follows that we can use such theory to classify some kinds of contacts between a vector field with $\s$ (\cite{ST}).



\begin{definition}
	Let $X$ be a $\mathcal{C}^\infty$ vector field. Then:
	\begin{enumerate}[i)]
		\item A  point $p\in\s$ such that $Xf(p)\neq 0$ is said to be \textbf{regular}. 
		\item A point $p\in\s$ such that $Xf(p)=0$ is said to be a \textbf{fold} of $X$ if $X^2f(p)\neq 0$.
		\item A point $p\in\s$ such that $Xf(p)=0$ is said to be a \textbf{cusp} of $X$ if $X^2f(p)=0$, $X^3f(p)\neq 0$ and $\det(\nabla f(p),\nabla Xf(p),\nabla X^2f(p))\neq 0$.
	\end{enumerate}

\end{definition}

\begin{figure}[h!]
	\centering
	\bigskip
	\begin{overpic}[width=10cm]{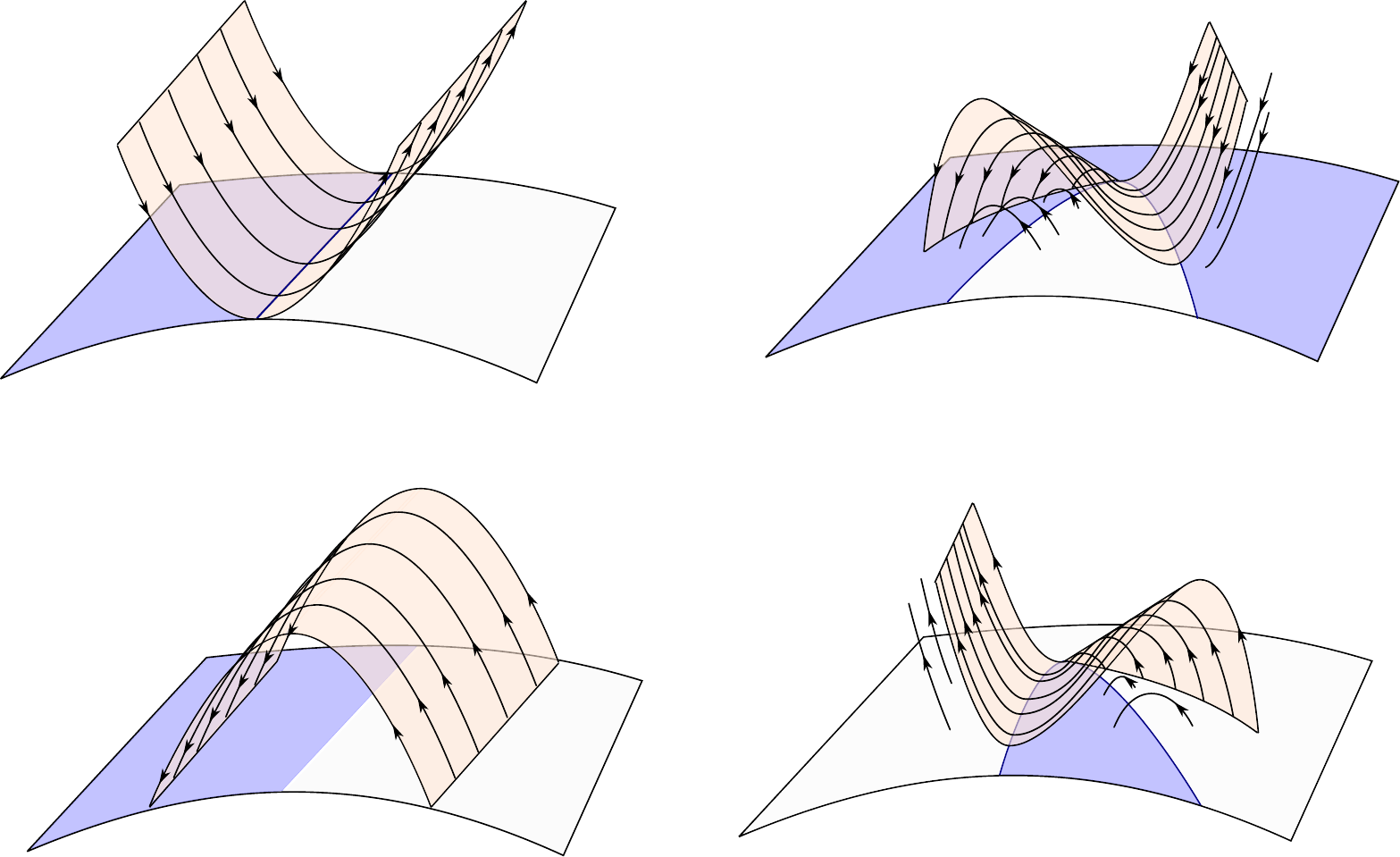}
			
		\end{overpic}
		\bigskip
		\caption{Fold (left) and cusp (right) singularities of $X$.}	\label{cf_fig}
	\end{figure}

It has been proved that, in dimension $3$, cusps and folds are generic singularities of vector fields in manifolds with boundary (see Figure \ref{cf_fig}), and in \cite{V}, the author has obtained a change of coordinates which brings the systems around such singularities in a very simple normal form. Below, we state a $\Cr$ version of the results contained in \cite{V} which is obtained in  \cite{matheus}.

\begin{theorem}[Vishik's Normal Form] \label{Vishik}
	Let $X\in\Xr$. If $p\in S_{X}$ is either a fold point or a cusp point of $X$  then there exist a neighborhood $V(p)$ of $p$ in $M$, a system of coordinates $(x_{1},x_{2},x_{3})$ at $p$ defined in $V(p)$ ($x_{i}(p)=0$, $i=1,2,3$) such that:
	\begin{enumerate}
		\item If $p$ is a fold point, then $X|_{V(p)}$ is a germ at $V(p)\cap \s$ of the vector field given by:
		\begin{equation}\label{foldV}
			\left\{ \begin{array}{l}
				\dot{x_{1}}=x_{2},\\
				\dot{x_{2}}=1,\\
				\dot{x_{3}}=0.\\
			\end{array}\right.
		\end{equation}
		\item If $p$ is a cusp point, then $X|_{V(p)}$ is a germ at $V(p)\cap \s$ of the vector field given by:
		\begin{equation}\label{cuspV}
			\left\{ \begin{array}{l}
				\dot{x_{1}}=x_{2},\\
				\dot{x_{2}}=x_{3},\\
				\dot{x_{3}}=1.\\
			\end{array}\right.
		\end{equation}	
		\item $\s$ is given by the equation $x_{1}=0$ in $V(p)$.
	\end{enumerate}
\end{theorem}

Thus, in the context of the tangential singularities, a point $p\in\s$ is said to be a \textbf{fold-regular singularity} (resp. regular-fold) of $Z=(X,Y)$ if $p$ is a fold of $X$ (resp. $Y$) and a regular point of $Y$ (resp. $X$). Analogously, a point $p\in\s$ is said to be a \textbf{cusp-regular singularity} (resp. regular-cusp) of $Z=(X,Y)$ if $p$ is a cusp of $X$ (resp. $Y$) and a regular point of $Y$ (resp. $X$). 

When we have a double tangency at a point $p$ ($Xf(p)=Yf(p)=0$), we also have some special points which appear in this work.
A point $p\in\s$ is said to be a \textbf{fold-fold singularity} provided that $S_X\pitchfork S_Y$ at $p$ and that $p$ is a fold point of both $X$ and $Y$. In this case, we have some flavors of fold-fold singularities:

	\begin{itemize}
	\item $p$ is a \textbf{visible fold-fold} if $X^{2}f(p)>0$ and $Y^{2}f(p)<0$;
	\item  $p$ is an \textbf{invisible-visible fold-fold} if $X^{2}f(p)<0$ and $Y^{2}f(p)<0$;	
	\item  $p$ is a \textbf{visible-invisible fold-fold} if $X^{2}f(p)>0$ and $Y^{2}f(p)>0$;			
	\item  $p$ is an \textbf{invisible fold-fold} if $X^{2}f(p)<0$ and $Y^{2}f(p)>0$, in this case, $p$ is also called a \textbf{T-singularity}.				
\end{itemize}

Finally, we define the main subject of this work. A \textbf{cusp-fold singularity} (resp. \textbf{fold-cusp singularity}) of $Z=(X,Y)$ is a point $p\in\s$ which is a cusp (resp. fold) of $X$ and a fold (resp. cusp) of $Y$, such that $S_X\pitchfork S_Y$ at $p$, see Figure \ref{cuspfold_fig}.

	\begin{figure}[h!]
	\centering
	\bigskip
	\begin{overpic}[width=10cm]{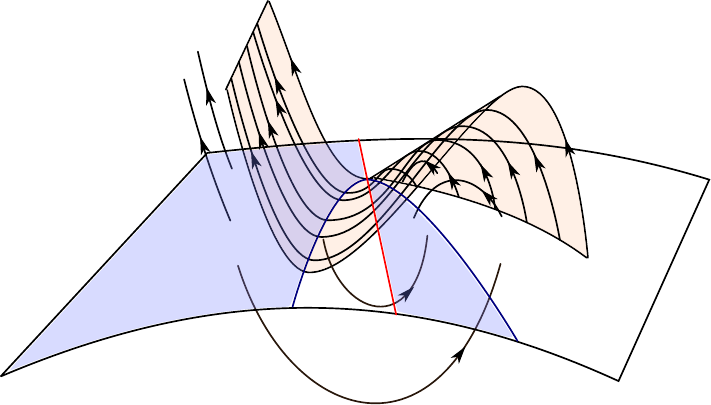}
			\put(50,34){{\scriptsize $p$}}			
			\put(53,8){{\scriptsize $S_Y$}}			
			\put(72,4){{\scriptsize $S_X$}}				
			
		\end{overpic}
		\bigskip
		\caption{Cusp-Fold Singularity.}	\label{cuspfold_fig}
	\end{figure} 

In order to analyze the behavior of perturbations of a Filippov system around a cusp-fold singularity, we introduce the concepts of some global objects appearing in this context.  

We say that $\gamma$ is a regular orbit of $Z=(X,Y)$ if it is a piecewise smooth curve such that $\gamma\cap M^{+}$ and $\gamma\cap M^{-}$ are unions of regular orbits of $X$ and $Y$, respectively, and  $\gamma\cap\s\subset\s^{c}$. 

We refer a \textbf{cycle} as a closed regular orbit $\Gamma$ of $Z$. Moreover, if $\Gamma\cap \s\neq \emptyset$, then $\Gamma$ is called a \textbf{crossing cycle} of $Z$. Following the concept of polycycle for planar Filippov systems introduced in \cite{Kamila-Otavio-2023}, we extend it for $3D$ Filippov systems. 

\begin{definition}\label{def_scyclegeneralized}
	A closed curve $\Gamma$ is said to be a \textbf{polycycle} of $Z=(X,Y)\in\Or$ if it is composed by a finite number of points, $p_1,p_2,\ldots,p_n$ and a finite number of regular oriented orbits  of $Z$, $\gamma_1,\gamma_2,\ldots,\gamma_n$, such that for each $1\leq i\leq n$, $\gamma_{i}$ starts at $p_i$ and ends at $p_{i+1}$, where $p_{n+1}=p_1$. Moreover:
	\begin{enumerate}[i)]
		\item $\Gamma$ is a $S^{1}$-immersion and it is oriented by increasing time along the regular orbits;
		\item if $p_{i}\in\s$ then it is a $\s$-singularity;
		\item if $p_{i}\in M^{\pm}$ then it is an equilibrium of either $X\big|_{M^+}$ or $Y\big |_{M^{-}}$;
		\item there exists a non-constant first return map induced by the flow of $Z$ defined in a $2$-dimensional surface $\sigma$ (called section) where $\Gamma\cap\sigma=\{q\}$;
		\item  Either $q\in\textrm{int}(\sigma)$ or $q\in\partial\sigma$.
	\end{enumerate}
	In particular, if $p_{i}\in\s$, for all $1\leq i\leq n$, then $\Gamma$ is said to be a \textbf{$\s$-polycycle}.
\end{definition}

In this paper, we are concerned in finding bifurcating $\s$-polycycles composed by a fold-regular singularity $p_1$ and a unique regular orbit $\gamma_1$, which will be simply refereed as a fold-regular polycycle.

%

In \cite{GT}, a fold-regular polycycle has been studied in generic one-parameter families $Z_{\alpha}$ of $3D$ Filippov systems, and under some generic conditions, they prove that the breaking of the polycycle of $Z_0$ when $\ag\neq0$ gives rise to either a CLC or a cycle with sliding segment, see Figure \ref{ccfig}. 

\begin{figure}[h!]
	\centering
		\begin{overpic}[width=15cm]{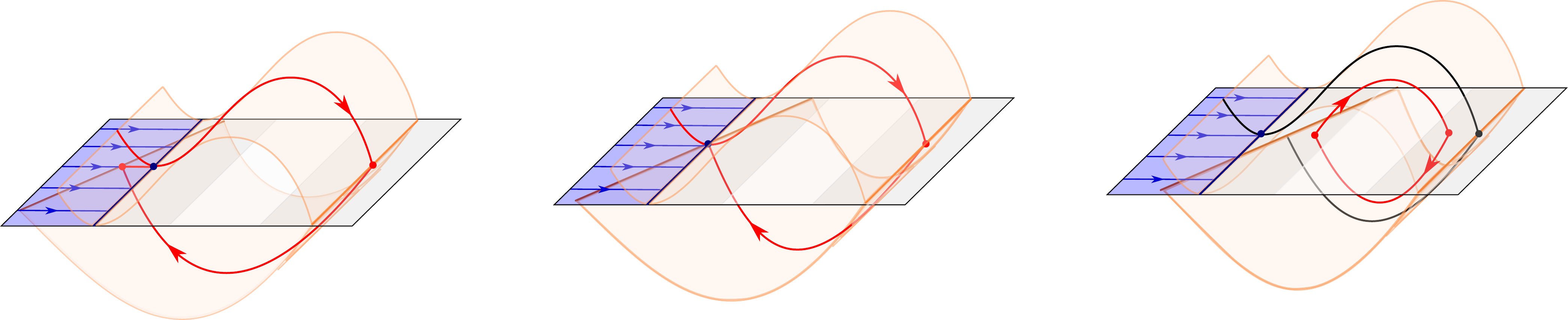}
		\put(46,-3){{\footnotesize $\ag=0$}}
		\put(48,16){{\footnotesize $\Gamma_0$}}																
		\put(11,-3){{\footnotesize $\ag<0$}}
		\put(12,14){{\footnotesize $\Gamma_\ag$}}			
		\put(82,-3){{\footnotesize $\ag>0$}}	
		\put(90,5){{\footnotesize $\Gamma_\ag$}}			
	\end{overpic}
	\bigskip
	\caption{A one-parameter family $Z_{\ag}$ presenting a fold-regular polycycle bifurcation at $\alpha=0$. For $\alpha<0$, $Z_{\ag}$ has a sliding cycle, and for $\ag>0$, $Z_{\ag}$ has a CLC.}	\label{ccfig}
\end{figure}
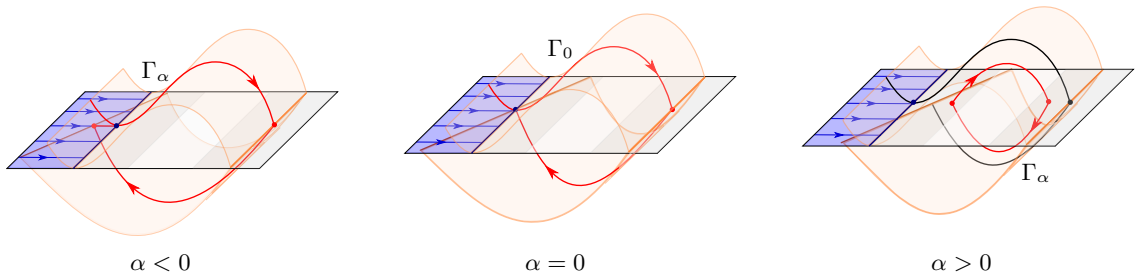

In order to identify such global phenomena locally bifurcating from a  cusp-fold singularity, we shall study the fixed points of some first return maps which correspond to zeros of a certain function called \textbf{displacement function}. In light of this, we enunciate below the classical Malgrange's Preparation Theorem which is a useful tool to detect such zeroes.  The proof can be found in \cite{MalgrangeProof}.

\begin{theorem}[Malgrange's Preparation Theorem] \label{malgrange}
Let \( f(t, x) \) be a \( C^\infty \) complex-valued function defined in a neighbourhood of the origin in \( \mathbb{R}^{n+1} \), here \( t \in \mathbb{R} \), \( x \in \mathbb{R}^n \), and let \( p > 0 \) be the first integer such that
\begin{equation}
\frac{\partial^p f }{\partial t^p}(0, 0) \neq 0.
\end{equation}
Then in a neighbourhood of the origin one has the factorization
\begin{equation}
f(t, x) = Q(t, x) P(t,x),
\end{equation}
where
\begin{equation}
P(t,x) = t^p + \sum_{j=1}^p \lambda_j(x) t^{p-j}
\end{equation}
and \( Q \) and \( \lambda_j \) are \( C^\infty \) complex-valued functions with \( Q(0, 0) \neq 0 \). If \( f \) is real there is such a factorization with \( Q \) and \( P \) real.
\end{theorem}


\section{Main Results}

First, we provide a normal form based on Theorem \ref{Vishik} which will be very useful to describe our main results below.

\begin{theorem}\label{NormalFormThm}
	Let $Z_\delta=(X_\delta,Y_\delta)$ be a $k$-dimensional unfolding of a Filippov system $Z_0$ having a cusp-fold singularity at $p\in\s$ in such a way that $p$ is a cusp of $X_0$ and a fold of $Y_0$, and assume that $X_\delta,Y_\delta$ have a $C^\infty$-dependence on $\delta$. There exist a time-rescaling of type $\tau=\pm t$ and a $C^\infty$ change of coordinates $\psi^\delta$ around $p$ such that $\psi^\delta(p)=0$, $\psi^\delta(\s)=\{(x,y,z);\ z=0\}$ and 
	\begin{equation}\label{normalformeq}
	\psi^\delta_*Z_\delta=\left\{\begin{array}{lcl}
		(-1,-(x+c), y), & if& z>0,  \\
		{\small \left(\displaystyle\sum\ag_{i,j,k}x^iy^jz^k, \displaystyle\sum\bg_{i,j,k}x^iy^jz^k, \displaystyle\sum\cg_{i,j,k}x^iy^jz^k\right)},& if&z<0,
	\end{array}\right.
\end{equation}
where $c,\ag_{i,j,k}, \bg_{i,j,k}, \cg_{i,j,k}$ are $\Cr$ $\delta$-functions such that $c(0)=0,\cg_{0,0,0}(\delta)\equiv0$ and $\sgn(\ag_{0,0,0}(\dg)\cg_{1,0,0}(\dg)+\bg_{0,0,0}(\dg)\cg_{0,1,0}(\dg))$ is constant for small $\delta$.	
	
\end{theorem}

The proof of this result can be found in Section \ref{normalformproof}. Since we are interesting in analyzing CLCs bifurcating  from a cusp-fold singularity of $Z_0=(X_0,Y_0)$, we shall restrict ourselves to the case where  $p$ is a cusp of $X_0$ and an invisible fold of $Y_0$. Nevertheless, some characteristics of the cusp-fold singularity are crucial to obtain some information on such  bifurcating invariant sets, so we establish some classes of invisible cusp-fold singularities that will be considered in this work.

\begin{definition}\label{deg1}
	We say that an invisible cusp-fold singularity at $p$ of a Filippov system  $Z_0=(X_0,Y_0)$ is of \textbf{degree }$1$ if  $Y_0X_0f(p)\neq 0$ and $S_{X_0}\pitchfork S_{Y_0}$ at $p$.
\end{definition}

We remark that since fold and cusp singularities are persistent in dimension $3$, and the conditions $Y_0X_0f(p)\neq 0$ and $S_{X_0}\pitchfork S_{Y_0}$ at $p$ in the definition above are persistent under small perturbations, it follows that the set
$$\mathcal{CF}_1(p)=\{Z_0\in\Or; p \textrm{ is an invisible cusp-fold singularity of degree }1\}$$
is a codimension-one submanifold of $\Or$. In fact, fixed $Z_0\in\mathcal{CF}_1$, from the local structural stability of fold and cusp points in dimension $3$, it follows that there exist small neighborhoods $\mathcal{U}$ of $Z_0$ and $U$ of $p$, such that for each $X\in\mathcal{U}$, there exists a unique cusp singularity $c_X$ of $X$ in $\mathcal{U}$ and $S_Y\cap U$ is composed by invisible fold points. In this case, translating the persistent intersecting point $S_X\cap S_Y$ to $p$, we can see that
$$\mathcal{CF}_1(p)\cap\mathcal{U}=\mathcal{G}^{-1}(0)\cap \mathcal{V},$$ where $\mathcal{V}$ is the open set given by
$$\mathcal{V}=\{Z=(X,Y)\in\Or;\ S_{X}\pitchfork S_{Y} \textrm{ at p },\ YXf(p)\neq 0 \},$$ and $0$ is a regular value of $\mathcal{G}:\mathcal{U}\rightarrow \rn{}$ given by $\mathcal{G}(X,Y)=\textrm{dist}(c_X,S_Y\cap U)$, where $c_X$ is the cusp point of $X$. More details can be found in \cite{Novaes_2018}.

Also, a generic unfolding of $Z_0=(X_0,Y_0)\in\mathcal{CF}_1(p)$ can be seen as a one-parameter unfolding which breaks the coincidence of the cusp singularity of $X_0$ and the  fold singularity of $Y_0$. The next result allows us to see that there is no unfolding of $Z_0\in\mathcal{CF}_1(p)$ presenting bifurcating CLCs.

\begin{theorem}\label{thmnonexistence}
	If $Z_0=(X_0,Y_0)\in\mathcal{CF}_1(p)$, then there exist neighborhoods $\mathcal{U}\subset\Or$ of $Z_0$ and $U\subset \rn{3}$ of $p$ such that, any $Z\in\mathcal{U}$ has no CLCs in $U$.
\end{theorem}

Geometrically, we notice that the generic condition $Y_0X_0f(p)\neq 0$ in Definition \ref{deg1} is equivalent to say that the vectors $Y_0(p)$ and $X_0(p)$ are not anti-collinear in $T_p\s$. This assumption rules out the existence of bifurcating limit cycles from a cusp-fold singularity, even when the connection between the fold and the cusp is broken (parameter $c\neq 0$ in the normal form). As a consequence, we have the nonexistence of bifurcating cycles from this singularity in generic one-parameter families. A detailed proof of Theorem \ref{nonexistenceproof}  is presented in Section \ref{nonexistenceproof}.


Finally, since the non anti-collinearity of $Y_0(p)$ and $X_0(p)$ prevents $Z_0=(X_0,Y_0)$ to present bifurcating CLCs, we see that the class of cusp-fold singularities which violates such condition should be investigated in order to find such bifurcating sets.

In order to define a new class of cusp-fold singularities, recall that, since $Z_0$ has an invisible cusp-fold singularity at $p$, it follows from Theorem \ref{NormalFormThm}, that there exists a change of variables around $p$ and a rescaling of time which bring $Z_0$ to
\begin{equation}\label{normalformeZ0}
	Z_0(x,y,z)=\left\{\begin{array}{lcl}
		(-1,-x, y), & if& z>0,  \\
		{\small \left(\displaystyle\sum\ag_{i,j,k}x^iy^jz^k, \displaystyle\sum\bg_{i,j,k}x^iy^jz^k, \displaystyle\sum\cg_{i,j,k}x^iy^jz^k\right)},& if&z<0,\end{array}\right.
\end{equation}
and we notice that $\bg_{0,0,0}=Y_0X_0f(p)$.

\begin{definition}\label{degdef}
	We say that an invisible cusp-fold singularity $p\in\s$ of a Filippov system $Z=(X_0,Y_0)$ is of \textbf{degree 2 with index $L_0$} if it satisfies the following conditions:
	\begin{enumerate}
					\item $S_{X_0}\pitchfork S_{Y_0}$ at $p$;		
		\item $Y_0X_0f(p)=0$;		
		\item 	 $L_0$ is the number given by
		\begin{equation}\label{Lformula}
			\begin{array}{lcl}
				
				L_0&=    & -\dfrac{4}{3\ag_{0,0,0}^2\cg_{1,0,0}}\left(-\alpha _{0,0,0} \gamma _{1,0,0} \left(\alpha _{1,0,0}-\beta _{0,1,0}-\beta _{2,0,0}+\gamma _{0,0,1}\right)\right. \vspace{0.3cm} \\
				
				&+  &\left. \beta _{1,0,0} \left(2 \gamma _{0,1,0}+\gamma _{2,0,0}\right)\right)+\alpha _{0,0,0}^2 \left(\gamma _{0,1,0}+\gamma _{2,0,0}\right) \vspace{0.3cm} \\ 
				
				&-& \left.\beta _{0,0,1} \gamma _{1,0,0}^2+\beta _{1,0,0}^2 \gamma _{0,1,0}+\beta _{1,0,0} \gamma _{1,0,0} \left(\gamma _{0,0,1}-\beta _{0,1,0}\right)\right);
			\end{array}
		\end{equation}
		\item $L_0\neq 0.$
	\end{enumerate}
\end{definition}

We notice that
$$\mathcal{CF}_2(p)=\{Z_0\in\Or; p \textrm{ is an invisible cusp-fold singularity of degree }2\}$$
is a codimension-two submanifold of $\Or$. In fact, fixed $Z_0\in\mathcal{CF}_2$, analogously to the previous class, there exist small neighborhoods $\mathcal{W}$ of $Z_0$ and $W$ of $p$, such that for each $X\in\mathcal{W}$, there exists a unique cusp singularity $c_X$ of $X$ in $\mathcal{W}$ and $S_Y\cap W$ is composed by invisible fold points, $S_{X}\pitchfork S_{Y}$ and if $q\in S_X\cap S_Y$ is a cusp-fold singularity then its index \eqref{Lformula} is nonzero. In this case, translating the persistent intersecting point $S_X\cap S_Y$ to $p$, we can see that
$$\mathcal{CF}_2(p)\cap\mathcal{W}=\mathcal{H}^{-1}(0,0),$$ where  $(0,0)$ is a regular value of $\mathcal{H}:\mathcal{W}\rightarrow \rn{2}$ given by $\mathcal{H}(X,Y)=(\textrm{dist}(c_X,S_Y\cap W),YXf(c_X))$, where $c_X$ is the cusp point of $X$. 

Now, we study the existence of CLCs bifurcating from a cusp-fold singularity of degree $2$. In order to do this, we shall notice that, for any $k$-dimensional unfolding $Z_\delta=(X_\delta,Y_\delta)$ of a Filippov system $Z_0$ having a cusp-fold singularity of degree $2$ at $p\in\s$, one can find via Theorem \ref{NormalFormThm} a coordinate system such that

\begin{equation}\label{normalzdelta}
	Z_{\dg}=\left\{\begin{array}{lcl}
		(-1,-(x+c(\dg)), y), & if& z>0,  \\
		{\small \left(\displaystyle\sum\ag_{i,j,k}(\dg)x^iy^jz^k, \displaystyle\sum\bg_{i,j,k}(\dg)x^iy^jz^k, \displaystyle\sum\cg_{i,j,k}(\dg)x^iy^jz^k\right)},& if&z<0,\end{array}\right.
\end{equation}
 where $c(\dg), \ag_{i,j,k}(\dg), \bg_{i,j,k}(\dg), \cg_{i,j,k}(\dg)$ are $\Cr$ functions on a small neighborhood $U\subset\rn{k}$ of the origin.
\begin{theorem}\label{Teo-CLC-Policiclo}
	Let $Z_\delta=(X_\delta,Y_\delta)$ be a $k$-dimensional unfolding of a Filippov system $Z_0$ having a cusp-fold singularity of degree $2$ of index $L_0$ at $p\in\s$ given by \eqref{normalzdelta}. There exists a germ of function $\bg_{0,0,0}^*:(\rn{k}\times\rn{},(0,0))\rightarrow(\rn{},0)$ given by
		$$\bg_{0,0,0}^*(\dg,Z)=\ag_{0,0,0}(\dg)Z+\dfrac{3L(\dg)\ag_{0,0,0}(\dg)}{2}Z^2+\er(Z^3),$$
		such that, if 
		
	\begin{enumerate}[i)]
		\item $\dfrac{1}{L(\dg)\ag_{0,0,0}(\dg)}(\bg_{0,0,0}(\dg)-\ag_{0,0,0}(\dg)c(\dg))>0$,
		\item $|\bg_{0,0,0}(\dg)-\ag_{0,0,0}(\dg)c(\dg)|\leq|\bg_{0,0,0}^*(\dg,c(\dg))-\ag_{0,0,0}(\dg)c(\dg)|$,
		\item $\cg_{1,0,0}(\dg)>0$ and $c(\dg)>0$ sufficiently small,
	\end{enumerate}
where $L(\dg)=L_0+\er(\dg)$, and $\cg_{1,0,0}(\dg)$, $\ag_{0,0,0}(\dg),\bg_{0,0,0}, c(\dg)$ are given by \eqref{normalzdelta}, then:  
\begin{enumerate}
	\item $Z_\dg$ has a unique (one-loop) CLC for every $(c(\dg),\bg_{0,0,0}(\dg))$ such that $|\bg_{0,0,0}(\dg)-\ag_{0,0,0}(\dg)c(\dg)|<|\bg_{0,0,0}^*(\dg,c(\dg))-\ag_{0,0,0}(\dg)c(\dg)|.$
	\item $Z_\dg$ has a unique (one-loop) critical crossing cycle (polycycle) passing through a fold-regular point for every $(c(\dg),\bg_{0,0,0}(\dg))$ such that $\bg_{0,0,0}(\dg)=\bg_{0,0,0}^*(\dg,c(\dg)).$
\end{enumerate}

\end{theorem}

\begin{remark}
	It is worth mentioning that, if $\nabla c(0)$ and $\nabla \bg_{0,0,0}(0)$ are linearly independent, then the associated unfolding realizes all the items of the previous theorem for some region of the $k$-dimensional parameter space. Also, since $c(\dg)$ is the bifurcation parameter responsible for ruling out the cusp-fold singularity and $\bg_{0,0,0}$ is the bifurcation parameter responsible for breaking the anti-collinearity of $X(0)$ and $Y(0)$, it follows that if $\nabla c(0)$ and $\nabla \bg_{0,0,0}(0)$ are linearly independent, then it is a generic unfolding of $Z_0\in\mathcal{CF}_2$.
\end{remark}

This result allows us to see that, for vector fields in  $\mathcal{CF}_2(p)$, we can detect a CLC birthing at a bifurcating Teixeira singularity (which is classically detected in the presence of the anti-collinearity hypothesis) which grows until reach the tangency line of the cusp at a fold-regular point, where it degenerates itself at a polycycle through such fold-singularity and then it disappears in the sequel. All such phenomena occur for small perturbations of the Filippov system and there are parameters which carries both CLC and polycycle to the cusp-fold singularity.

\section{Real and Toy-Model Applications} 

\subsection{Analysis of the Semi-Linear part of the Normal Form around a Codimension Two Cusp-Fold Singularity}\label{Toy-Model}
We consider a simple but non-trivial example given by the system
\begin{equation}
	\begin{bmatrix}
		\dot{x}\\
		\dot{y}\\
		\dot{z}
	\end{bmatrix}=\left\{\begin{matrix}
		Y(x,y,z)=\begin{bmatrix}
			-1\\
			-\beta\\
			-x-by
		\end{bmatrix},\;\;\;\text{if}\;\;\;z<0,\\
		X(x,y,z)=\begin{bmatrix}
			1\\
			x+\alpha\\
			-y
		\end{bmatrix},\;\;\;\;\;\;\;\text{if}\;\;\;z>0,
	\end{matrix}\right.\label{Ex1}
\end{equation}
where $\alpha,\beta,b$ are parameters and  $\Sigma=\{(x,y,z)\in\mathbb{R}^3: f(x,y,z)=z=0\}$ is the switching boundary.

Since $Xf(x,y,0)=-y$ and $Yf(x,y,0)=-x-by$, then sliding and escaping dynamics occur in the regions $\Sigma^{s}=\{(x,y,0):y>0\;\;\text{and}\;\;x<-by\}$ and $\Sigma^{e}=\{(x,y,0):y<0\;\;\text{and}\;\;x>-by\}$, respectively, while that the crossing regions are $\Sigma^{c}_-=\{(x,y,0):y>0\;\;\text{and}\;\;x>-by\}$ and $\Sigma^{c}_+=\{(x,y,0):y<0\;\;\text{and}\;\;x<-by\}$. There are two tangency lines that intersect transversely at $(0,0,0)$, namely $S_X=\{(x,y,0): y=0\}$ for the vector field $X$ and $S_Y=\{(x,y,0): x+by=0\}$ for the vector field $Y$. For the study that follows, we want $S_Y$ to have only invisible folds and $S_X$ to have a cusp singularity.

Tangency points in $S_Y$ are invisible folds whenever $1+b\beta>0$ (since $Y^2f(S_Y)=1+b\beta$). In $S_X$, there are invisible folds for $x>-\alpha$ (since $X^2f(S_X)=-\alpha-x$), visible folds for $x<-\alpha$, and cusp singularity at $\mathbf{q}=(-\alpha,0,0)\in S_X$ (since $X^3f(\mathbf{q})=-1$ and $\text{det}(\nabla f,\nabla Xf,\nabla X^2f)=-1$). The double tangency at $(0,0,0)$ is a T-singularity whenever $\alpha>0$ and $1+b\beta>0$, becoming  a cusp-fold (invisible) when $\alpha=0$ and $1+b\beta>0$.

\begin{remark}
	When $\ag=\bg=0$ and $b\neq0$, the double-tangency singularity of \eqref{Ex1} is an invisible cusp-fold singularity of degree 2 with index $L_0=-4b/3$. 
\end{remark}

Vector fields $Y$ and $X$ are anticollinear at $(0,0,0)$ for $\alpha=\beta$. In what follows, we study the existence and stability of CLCs in system \eqref{Ex1} that bifurcate from $(0,0,0)$ for $\alpha>\beta$ or $\alpha<\beta$.

\begin{prop}
	Assume in system \eqref{Ex1} that $b\neq 0$, $\beta>0$ and $1+b\beta>0$. So, there is a single CLC that appears for:
	\begin{itemize}
		\item[a)] $b>0$ and $\beta<\alpha<\beta(1+3b\beta)$; or
		\item[b)] $b<0$ and $\beta(1+3b\beta)<\alpha<\beta$.
	\end{itemize}   
	In addition, this CLC becomes a policycle for $\alpha=\beta(1+3b\beta)$.
\end{prop}
\begin{proof}
	We take the trajectory of system \eqref{Ex1} starting at $\mathbf{x}_0=(x_0,y_0,0)\in\Sigma$ and returning to $\Sigma$ for the first time at point $\mathbf{x}_1=(x_1,y_1,0)$. The coordinates of the point $\mathbf{x}_1$ as a function of $\mathbf{x}_0$ are easily obtained, namely
	\begin{equation}
		\begin{bmatrix}
			x_1 \\ y_1
		\end{bmatrix}=P_-(x_0,y_0)=\dfrac{1}{1+b\beta}\begin{bmatrix}
			(b\beta-1)x_0-2by_0\\
			-2\beta x_0+(1-b\beta)y_0
		\end{bmatrix},
	\end{equation}
	where $x_0>-by_0$ and $1+b\beta>0$.
	
	Next, we take the trajectory of system \eqref{Ex1} starting at $\mathbf{x}_1=(x_1,y_1,0)\in\Sigma$ and returning to $\Sigma$ for the first time at point $\mathbf{x}_2=(x_2,y_2,0)$. The coordinates of the point $\mathbf{x}_2$ are obtained as a function of $(t,x_1,y_1)$,  namely
	\begin{equation}
		\begin{bmatrix}
			x_2 \\ y_2
		\end{bmatrix}=P_+(t,x_1,y_1)=\begin{bmatrix}
			t + x_1\\
			1/2 (t^2 + 2 (x_1+\alpha)t + 2 y_1)
		\end{bmatrix},
	\end{equation}
	where $y_1<0$ and $t\geq 0$ (flight time of the trajectory) is solution of the equation
	\begin{equation}
		g(t,x_1,y_1)=-t^2  - 3 t (x_1 + \alpha)- 6 y_1=0.
	\end{equation}
	
	A CLC appears when we have $\mathbf{x}_2=\mathbf{x}_0\in\Sigma_-^c$ and the coordinates of its crossing points $\mathbf{x}_1\in\Sigma_+^c$ and $\mathbf{x}_0\in\Sigma_-^c$ with $\Sigma$ satisfy the equations $(x_1,y_1)=P_-(x_0,y_0)$, $(x_0,y_0)=P_+(t,x_1,y_1)$ and $g(t,x_1,y_1)=0$. 
	We solve such equations for the unknowns $(x_1,y_1,x_0,y_0,\alpha)$ and obtain a solution dependent on $t\geq 0$ given by
	\begin{align}
		x_1&=\dfrac{-t(6+bt)}{12},\;\;\;y_1=\dfrac{t(t-6\beta)}{12},\\
		x_0&=\dfrac{t(6-bt)}{12},\;\;\;y_0=\dfrac{t(t+6\beta)}{12},\\
		\alpha &=\dfrac{b}{12}t^2+\beta.\label{eq_beta}
	\end{align}
	We know that for $\alpha=\beta$ the vector fields $Y$ and $X$ are anticollinear at $(0,0,0)$, and we look for CLCs that bifurcate from this point for $\alpha>\beta$ or $\alpha<\beta$. Therefore, based on equation \eqref{eq_beta}, a necessary condition for the existence of CLCs is $b\neq 0$. For the obtained solution to make sense, that is, $\mathbf{x}_0\in\Sigma^c_-$ and $\mathbf{x}_1\in\Sigma^c_+$, we must have $y_0>0$, $x_0+by_0>0$, $y_1<0$ and $x_1+by_1<0$, from where we get that $\beta>0$, $1+b\beta>0$ and $0<t<6\beta$. In addition, for all $t\in (0,6\beta)$ the function $\alpha=\alpha(t)$ is increasing if $b>0$ (see Figure \ref{Fig1-Ex1-Setbifurcation1}) or decreasing if $b<0$ (see Figure \ref{Fig1-Ex1-Setbifurcation2}). Then, there is a different $\alpha$ for each $t\in (0,6\beta)$, implying the uniqueness of the CLC.

	Therefore, there is a single CLC for each $\alpha$ in the range $(\beta,\beta(1+3b\beta))$ if $b>0$ and $\beta>0$, or in the range $(\beta(1+3b\beta),\beta)$ if $b<0$ and $0<\beta<-\dfrac{1}{b}$. Moreover, for $\alpha=\beta$ the CLC collapses with the T-singularity (i.e., $\mathbf{x}_0=\mathbf{x}_1=(0,0,0)$) and for $\alpha=\beta(1+3b\beta)$ the CLC becomes a polycycle that intersects $\Sigma$ at points $\mathbf{x}_0=(3\beta(1-b\beta),6\beta^2,0)\in\Sigma^c_-$ and $\mathbf{x}_1=(-3\beta(1+b\beta),0,0)\in S_X$.
\end{proof}

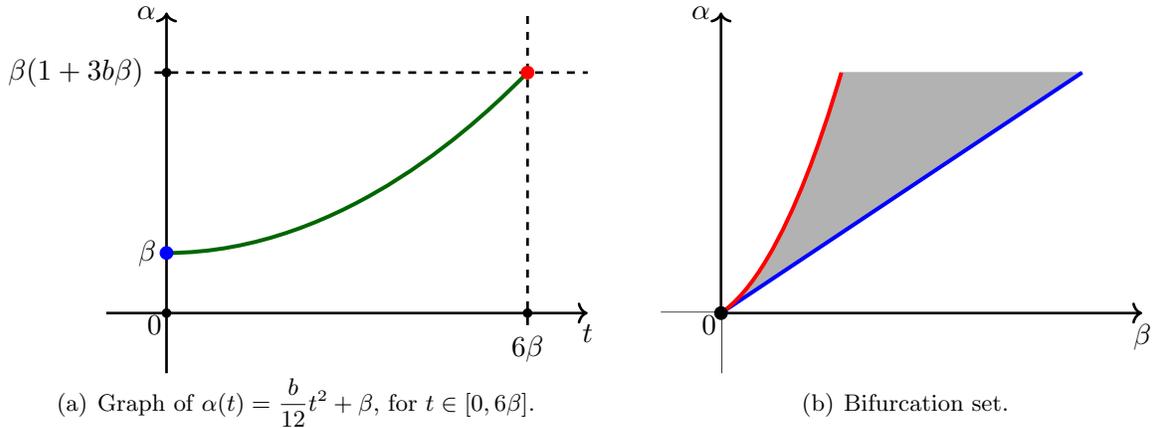
\begin{figure}[h]
	\begin{center}
		\subfigure[Graph of $\alpha(t)=\dfrac{b}{12}t^2+\beta$, for $t\in \text{[}0, 6\beta\text{]}$.]{
			\begin{tikzpicture}[scale = 0.8]
				\draw (-0.2,-0.2)node[]{$0$};
				\draw[line width=1pt,->] (-1,0) to (7,0)node[below]{$t$};
				\draw[line width=1pt,->] (0,-1) to (0,5)node[left]{$\alpha$};
				\draw[line width=1pt, dashed] (6,-0.2)node[below]{$6\beta$} to (6,5);
				\draw[line width=1pt, dashed] (-0.2,4)node[left]{$\beta(1+3b\beta)$} to (7,4);
				\draw[line width=1.5pt,black!60!green] (0,1) parabola (6,4); 
				\filldraw[blue] (0,1) circle (3pt) node[left,black]{$\beta$};
				\filldraw[red] (6,4) circle (3pt);
				\filldraw[black] (0,0) circle (2pt);
				\filldraw[black] (0,4) circle (2pt);
				\filldraw[black] (6,0) circle (2pt);
			\end{tikzpicture}\label{Fig1-Ex1-Setbifurcation1}}\hspace{0.5cm}
		\subfigure[Bifurcation set.]{
			\begin{tikzpicture}[scale = 0.8]
				\draw[line width=1pt,->] (-1,0) to (7,0)node[below]{$\beta$};
				\draw[line width=1pt,->] (0,-1) to (0,5)node[left]{$\alpha$};
				\filldraw[opacity=0.5,black!60!white] (-0.5,-0.166667) parabola (2,4) to (6,4) to (0,0); 
				\draw[line width=1.5pt, blue] (0,0) to (6,4);
				\draw[line width=1.5pt,red] (-0.5,-0.166667) parabola (2,4); 
				\filldraw[white] (0,0) to (0,-1) to (-1,-1) to (-1,0) to (0,0);
				\draw (-0.2,-0.2)node[]{$0$};
				\filldraw[black] (0,0) circle (3pt);
			\end{tikzpicture}\label{Fig2-Ex1-Setbifurcation1}}
		\caption{Existence and bifurcations of CLCs assuming $b>0$. From (a) we observe that there is a single CLC for each $\alpha$ in $(\beta,\beta(1+3b\beta))$. In (b) is shown the bifurcation set, where the painted region contains the points $(\beta,\alpha)$ so that the system has a CLC, the blue line is the bifurcation branch for the birth of a CLC from the T-singularity (TS-bifurcation) and the red curve is the bifurcation branch where the CLC becomes a Polycycle. }\label{Fig-Ex1-Setbifurcation1}
	\end{center}
\end{figure}

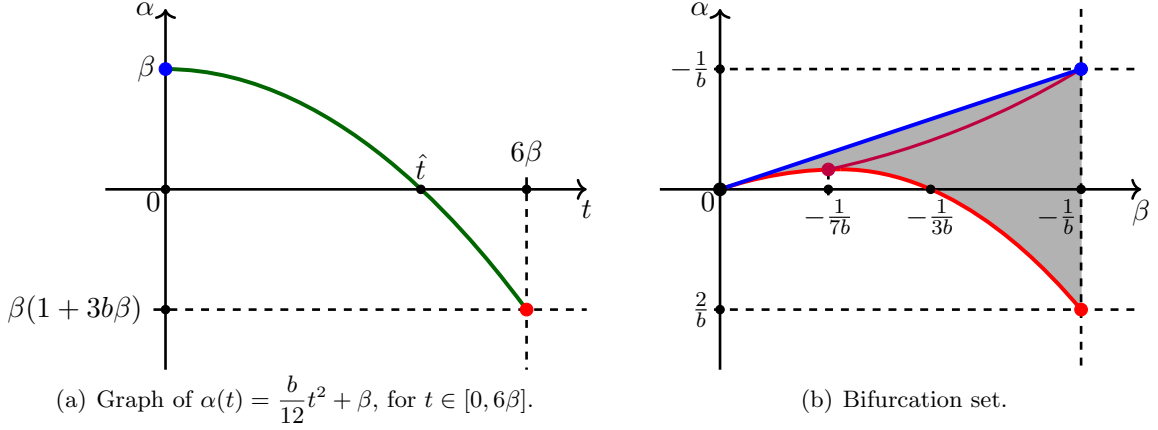
\begin{figure}[!h]
	\begin{center}
		\subfigure[Graph of $\alpha(t)=\dfrac{b}{12}t^2+\beta$, for $t\in \text{[}0, 6\beta\text{]}$.]{
			\begin{tikzpicture}[scale = 0.8]
				\draw (-0.2,1.8)node[]{$0$};
				\draw[line width=1pt,->] (-1,2) to (7,2)node[below]{$t$};
				\draw[line width=1pt,->] (0,-1) to (0,5)node[left]{$\alpha$};
				\draw[line width=1pt, dashed] (6,2.2)node[above]{$6\beta$} to (6,-1);
				\draw[line width=1pt, dashed] (-0.2,0)node[left]{$\beta(1+3b\beta)$} to (7,0);
				\draw[line width=1.5pt,black!60!green] (0,4) parabola (6,0); 
				\filldraw[blue] (0,4) circle (3pt) node[left,black]{$\beta$};
				\filldraw[red] (6,0) circle (3pt);
				\filldraw[black] (0,0) circle (2pt);
				\filldraw[black] (0,2) circle (2pt);
				\filldraw[black] (6,2) circle (2pt);
				\filldraw[black] (4.24264,2)node[above]{$\hat{t}$} circle (2pt);
			\end{tikzpicture}\label{Fig1-Ex1-Setbifurcation2}}\hspace{0.5cm}
		\subfigure[Bifurcation set.]{
			\begin{tikzpicture}[scale = 0.8]
				\filldraw[opacity=0.5,black!60!white] (0,2) to (6,4) to (6,0) to (2,7/3) parabola (0,2); 
				\filldraw[white] (2,7/3) parabola (6,0) to (0,0);
				\draw[line width=1pt,->] (-1,2) to (7,2)node[below]{$\beta$};
				\draw[line width=1pt,->] (0,-1) to (0,5)node[left]{$\alpha$};
				
				\draw[line width=1.5pt,red] (2,7/3) parabola (0,2); 
				\draw[line width=1.5pt,red] (2,7/3) parabola (6,0); 
				\draw[line width=1pt, dashed] (6,5) to (6,-1);
				\draw[line width=1pt, dashed] (-0.1,4) to (7,4);
				\draw[line width=1pt, dashed] (-0.1,0) to (7,0);
				\draw (-0.2,1.8)node[]{$0$};
				
				\filldraw[black] (0,4)node[left]{$-\frac{1}{b}$} circle (2pt);
				\filldraw[black] (0,0)node[left]{$\frac{2}{b}$} circle (2pt);
				\filldraw[black] (6,2) circle (2pt);
				\filldraw[black] (3.5,2)node[below]{$-\frac{1}{3b}$} circle (2pt);
				\draw (5.6,2)node[below]{$-\frac{1}{b}$};
				\draw[line width=1pt, dashed] (1.8,1.9) to (1.8,2.33);
				\filldraw[black] (1.8,2)node[below]{$-\frac{1}{7b}$} circle (2pt);
				\filldraw[purple] (1.8,2.33) circle (3pt);
				\draw[line width=1.2pt,purple] (1.8,2.33) .. controls (4,2.8) and (5.5,3.7) .. (6,4);
				\draw[line width=1.5pt, blue] (0,2) to (6,4);
				\filldraw[black] (0,2) circle (3pt);
				\filldraw[blue] (6,4) circle (3pt);
				\filldraw[red] (6,0) circle (3pt);
			\end{tikzpicture}\label{Fig2-Ex1-Setbifurcation2}}
		\caption{Existence and bifurcations of CLCs assuming $b<0$. From (a) we observe that there is a single CLC for each $\alpha$ in $(\beta(1+3b\beta),\beta)$. For this case,  $\hat{t}=2\sqrt{\frac{3\beta}{-b}}< 6\beta$ because we consider $-\frac{1}{3b}<\beta<-\frac{1}{b}$. In (b) is shown the bifurcation set, where the painted region contains the points $(\beta,\alpha)$ so that the system has a CLC, including on the purple curve. Blue line is the bifurcation branch for the birth of a CLC from the T-singularity (TS-bifurcation), the red curve is the bifurcation branch where the CLC becomes a Polycycle, and the purple curve is the bifurcation branch where the CLC changes stability, being stable above this curve and unstable below it.}\label{Fig-Ex1-Setbifurcation2}
	\end{center}
\end{figure}

Figures \ref{Fig-Ex1-Setbifurcation1} and \ref{Fig-Ex1-Setbifurcation2}  show the graph of the $\alpha(t)$ function in (a) and the bifurcation set in the $(\beta,\alpha)$-plane in (b), for the cases $b>0$ and $b<0$ respectively. The region painted in Figures \ref{Fig2-Ex1-Setbifurcation1} and \ref{Fig2-Ex1-Setbifurcation2} contains the points $(\beta,\alpha)$ for which system \eqref{Ex1} presents a CLC in its phase portrait. The blue and red curves are bifurcation branches associated with the bifurcation at the T-singularity (with equation $\alpha=\beta$, for all $\beta>0$ if $b>0$ or $0<\beta<-\dfrac{1}{b}$ if $b<0$) and the emergence of the polycycle (with equation $\alpha=\beta(1+3b\beta)$, for all $\beta>0$ if $b>0$ or $0<\beta<-\dfrac{1}{b}$ if $b<0$), respectively.

In order to study the stability of the CLC we use the Jacobian matrix of the first return map defined in $\Sigma^c_-$ and give by $P_+(t,P_-(x_0,y_0))$, namely $J=DP_+(t,x_1,y_1)\cdot DP_-$ with $DP_-=\dfrac{\partial P_-(x_0,y_0)}{\partial (x_0,y_0)}$ and  $DP_+(t,x_1,y_1)=\dfrac{\partial P_+(t,x_1,y_1)}{\partial (x_1,y_1)}-\dfrac{\partial P_+(t,x_1,y_1)}{\partial t}\dfrac{\partial g(t,x_1,y_1)}{\partial (x_1,y_1)}\left(\dfrac{\partial g(t,x_1,y_1)}{\partial t}\right)^{-1}$. Determinant of $J$ at the fixed point $(x_1(t),y_1(t),\alpha(t))$ associated with CLC is $\delta(t)=\text{det}(J)=1-\dfrac{2t}{t+6\beta}$. In this case, $\delta(t)$ is decreasing in $0\leq t\leq 6\beta$ and $0\leq\delta(t)\leq 1$. Therefore, the CLC is either stable (focus or node type) or unstable of the saddle type. 

Let $\tau(t)$ the trace of $J$ at the fixed point $(x_1(t),y_1(t),\alpha(t))$ associated with CLC. 
This fixed point is a saddle if 
$$\delta(t)-\tau(t)+1=-\dfrac{4bt^2}{(t+6\beta)(1+b\beta)}<0,$$ 
that is, for $b>0$ and $0<\beta<\alpha<\beta(1+3b\beta)$. On the other hand, this fixed point is stable with focus or node dynamics if $\delta(t)-\tau(t)+1>0$ and
$$\delta(t)+\tau(t)+1=\dfrac{4(bt^2+6b\beta^2+6\beta)}{(t+6\beta)(1+b\beta)}=\dfrac{24(2(\alpha-\beta)+b\beta^2+\beta)}{(t+6\beta)(1+b\beta)}>0,$$ 
that is, for $b<0$, $0<\beta<-\dfrac{1}{b}$ and $\text{max}\left[\beta(1+3b\beta),\beta(1-b\beta)/2\right]<\alpha<\beta$.
And, this fixed point is a saddle with negative eigenvalues if $\delta(t)+\tau(t)+1<0$, that is, for $b<0$, $0<\beta<-\dfrac{1}{b}$ and $\beta(1+3b\beta)<\alpha<\beta(1-b\beta)/2$. 

Based on the CLC stability results described above, we conclude that the CLC for the case $b>0$ is always unstable of the saddle type, while for the case $b<0$ the CLC is stable with focus or node dynamics if $0<\beta<-\dfrac{1}{b}$ and $\text{max}\left[\beta(1+3b\beta),\beta(1-b\beta)/2\right]<\alpha<\beta$ (painted region above the purple curve shown in Figure \ref{Fig2-Ex1-Setbifurcation2}), but it is unstable of the saddle type with negative eigenvalues if $0<\beta<-\dfrac{1}{b}$ and $\beta(1+3b\beta)<\alpha<\beta(1-b\beta)/2$ (painted region below the purple curve shown in Figure \ref{Fig2-Ex1-Setbifurcation2}). For points $(\beta,\alpha)$ such that $\alpha=\dfrac{\beta(1-b\beta)}{2}$ and $0<-\dfrac{1}{7b}<\beta<-\dfrac{1}{b}$ (on the purple curve shown in Figure \ref{Fig2-Ex1-Setbifurcation2}) the CLC is non-hyperbolic with an eigenvalue equal to $-1$, indicating a possible period-doubling bifurcation of this CLC. Proof of such a bifurcation, as well as the existence of period doubling cascades of CLCs leading the system to chaotic behavior, requires a more detailed analysis, which we leave for a future work.

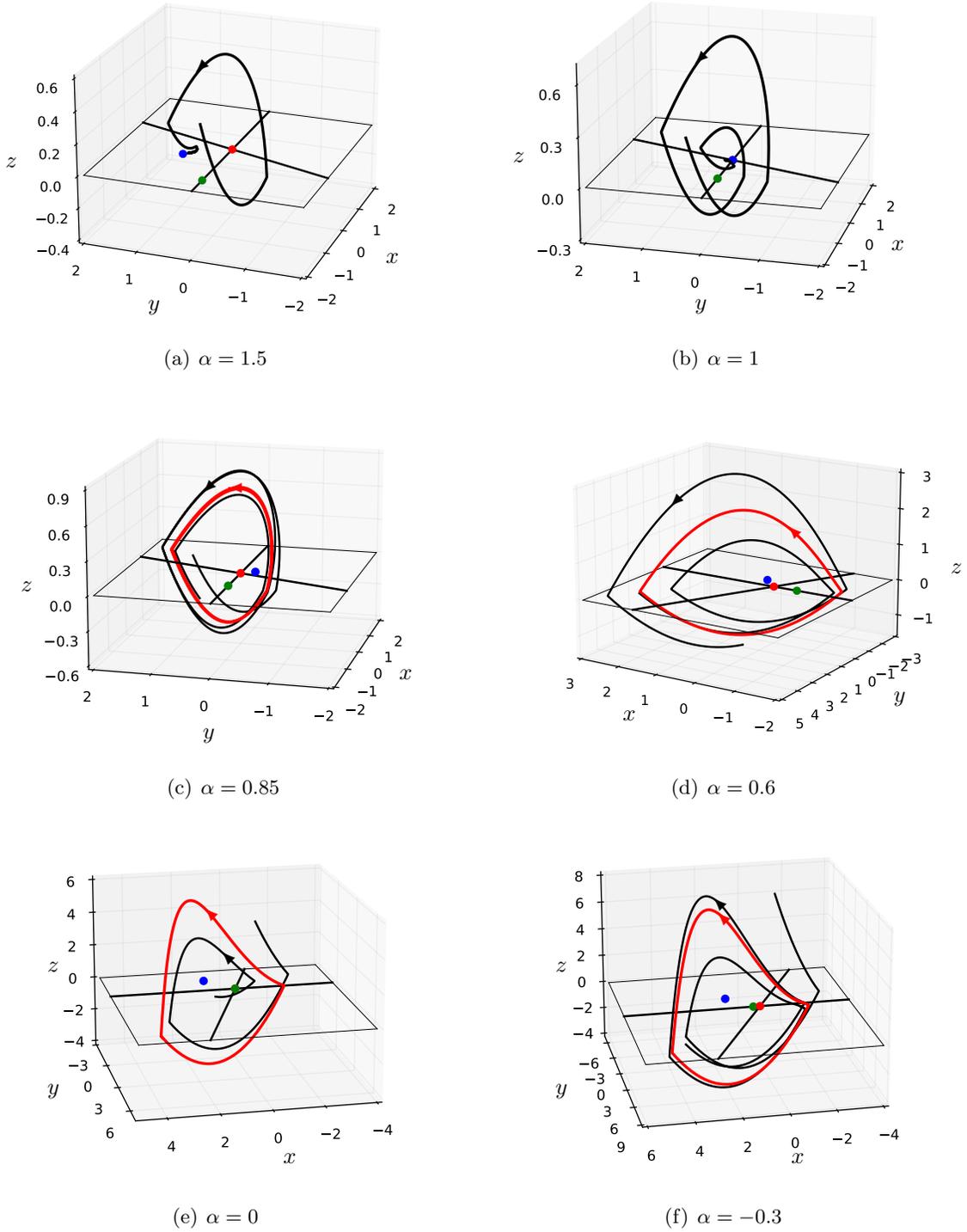
\begin{figure}[!h]
	\begin{center}
		\subfigure[$\alpha=1.5$]{
			\begin{overpic}[scale=0.6,clip,trim=2cm 1cm 2cm 2cm]{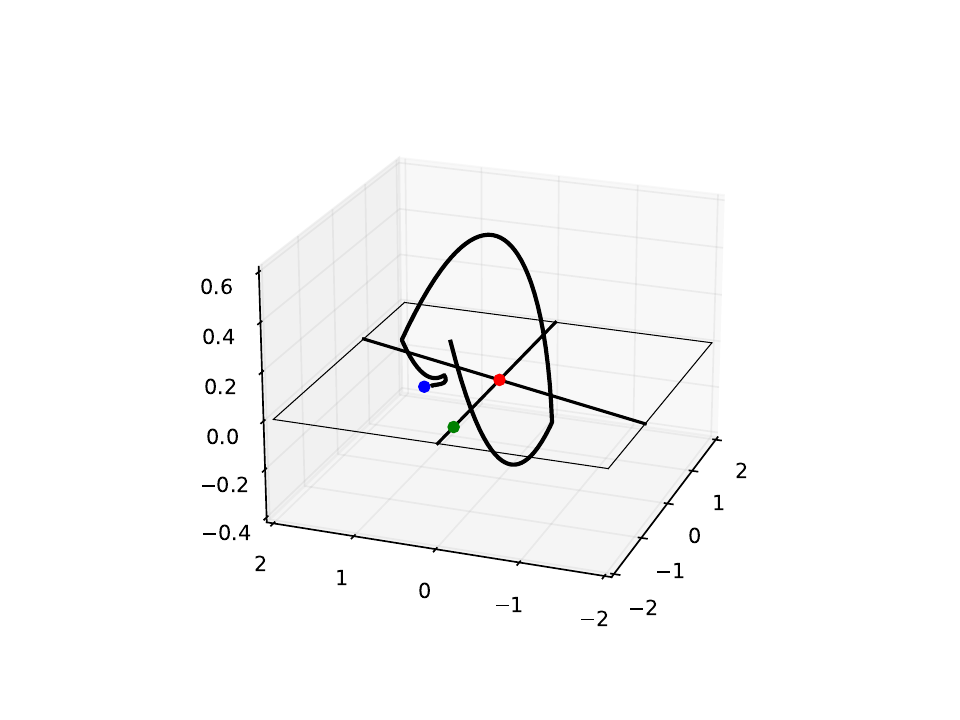}\label{Fig1-Ex1-CuspFold}
				\put(47,57){\color{black}\thicklines\vector(-1,-1){2}}
				\put(85,15){\small$x$}
				\put(35,5){\small$y$}
				\put(5,35){\small$z$}
		\end{overpic}}
		\subfigure[$\alpha=1$]{
			\begin{overpic}[scale=0.62,clip,trim=2cm 1cm 2cm 2cm]{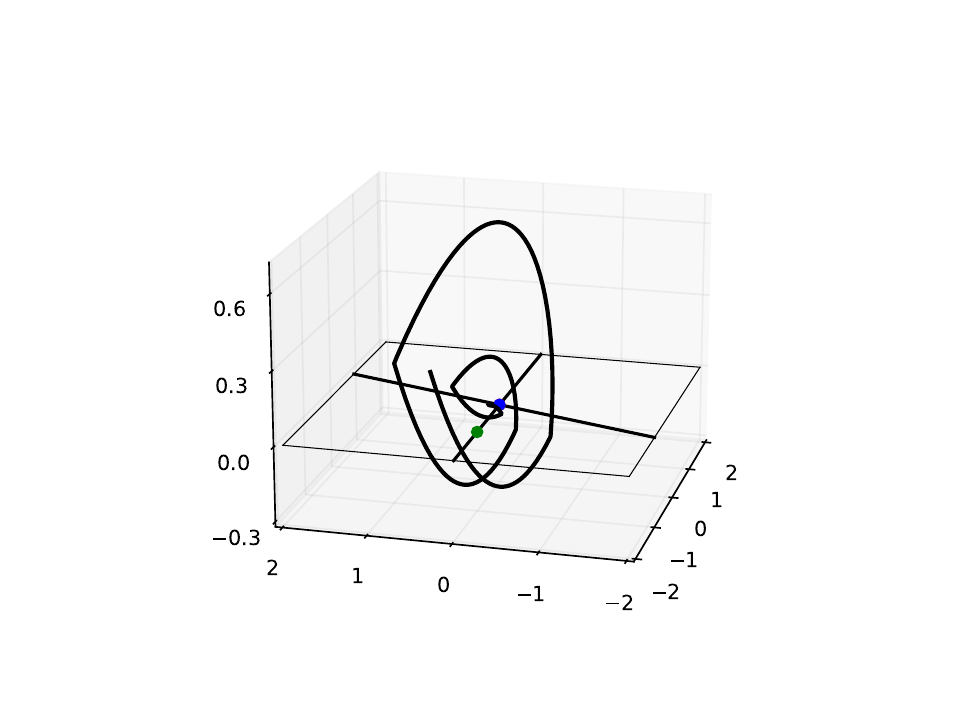}\label{Fig2-Ex1-CuspFold}
				\put(47,57){\color{black}\thicklines\vector(-1,-1){2}}
				\put(85,17){\small$x$}
				\put(45,5){\small$y$}
				\put(8,35){\small$z$}
		\end{overpic}}
		\subfigure[$\alpha=0.85$]{
			\begin{overpic}[scale=0.62,clip,trim=2cm 1cm 2cm 2cm]{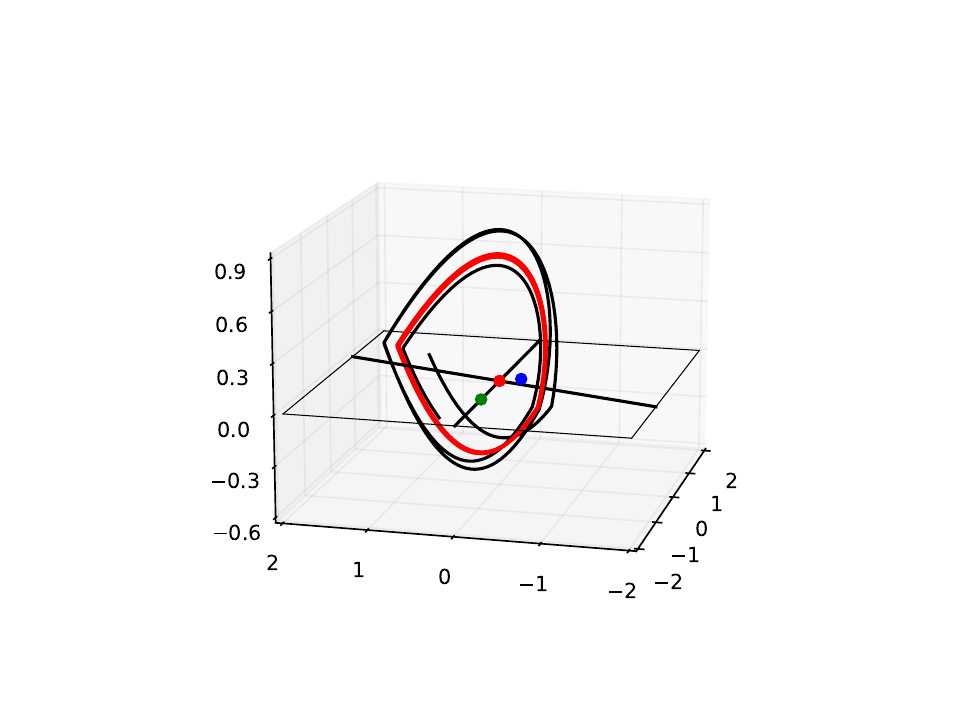}\label{Fig3-Ex1-CuspFold}
				\put(47,57){\color{black}\thicklines\vector(-1,-1){2}}
				\put(51.15,56.05){\color{red}\thicklines\vector(-1,0){1}}
				\put(85,17){\small$x$}
				\put(45,5){\small$y$}
				\put(8,35){\small$z$}
		\end{overpic}}
		\subfigure[$\alpha=0.6$]{
			\begin{overpic}[scale=0.6,clip,trim=2cm 1cm 2cm 2cm]{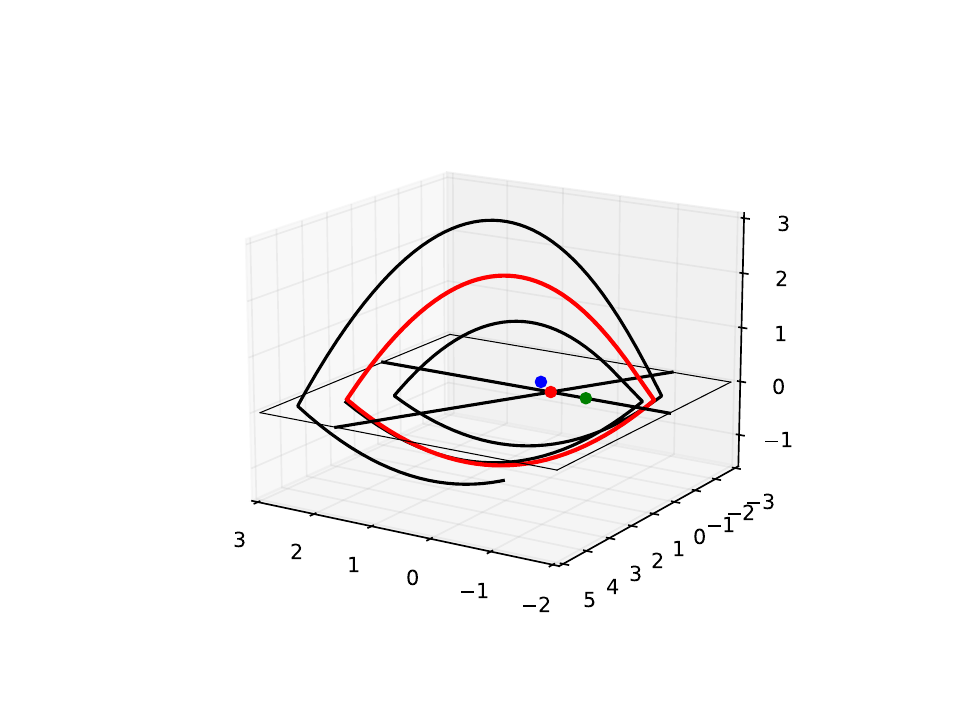}\label{Fig4-Ex1-CuspFold}
				\put(40,56){\color{black}\thicklines\vector(-1,-1){2}}
				\put(64.8,48){\color{red}\thicklines\vector(-1,1){2}}
				\put(28,9){\small$x$}
				\put(85,13){\small$y$}
				\put(97,40){\small$z$}
		\end{overpic}} 
		\subfigure[$\alpha=0$]{
			\begin{overpic}[scale=0.62,clip,trim=2cm 1cm 2cm 2cm]{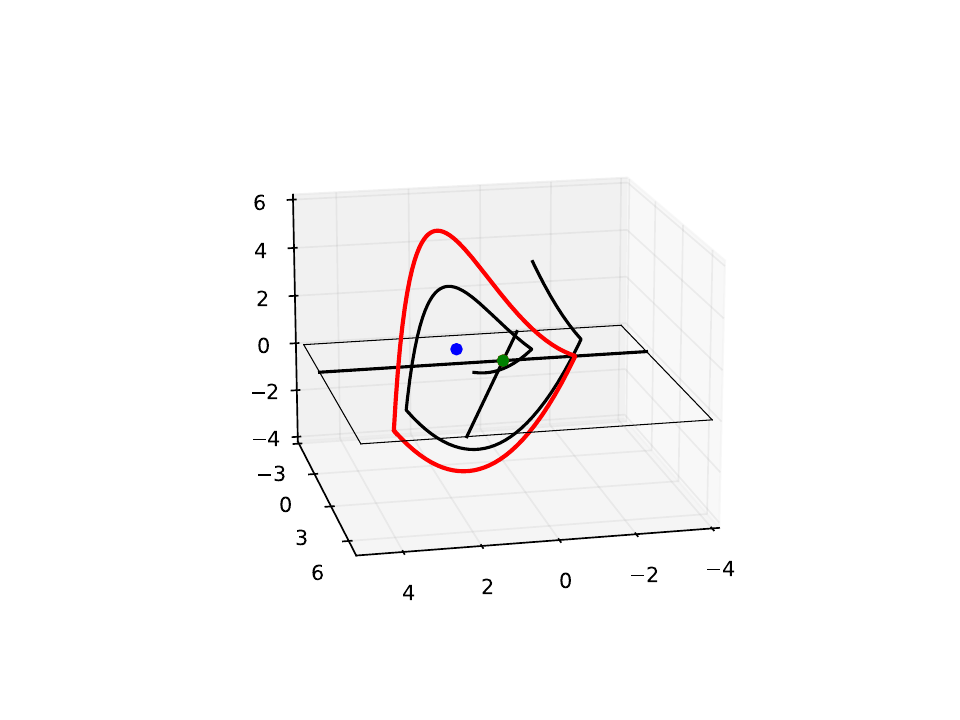}\label{Fig5-Ex1-CuspFold}
				\put(52.5,47){\color{black}\thicklines\vector(-1,1){2}}
				\put(49,56){\color{red}\thicklines\vector(-1,1){2}}
				\put(63,6){\small$x$}
				\put(15,20){\small$y$}
				\put(15,45){\small$z$}
		\end{overpic}}
		\subfigure[$\alpha=-0.3$]{
			\begin{overpic}[scale=0.62,clip,trim=2cm 1cm 2cm 2cm]{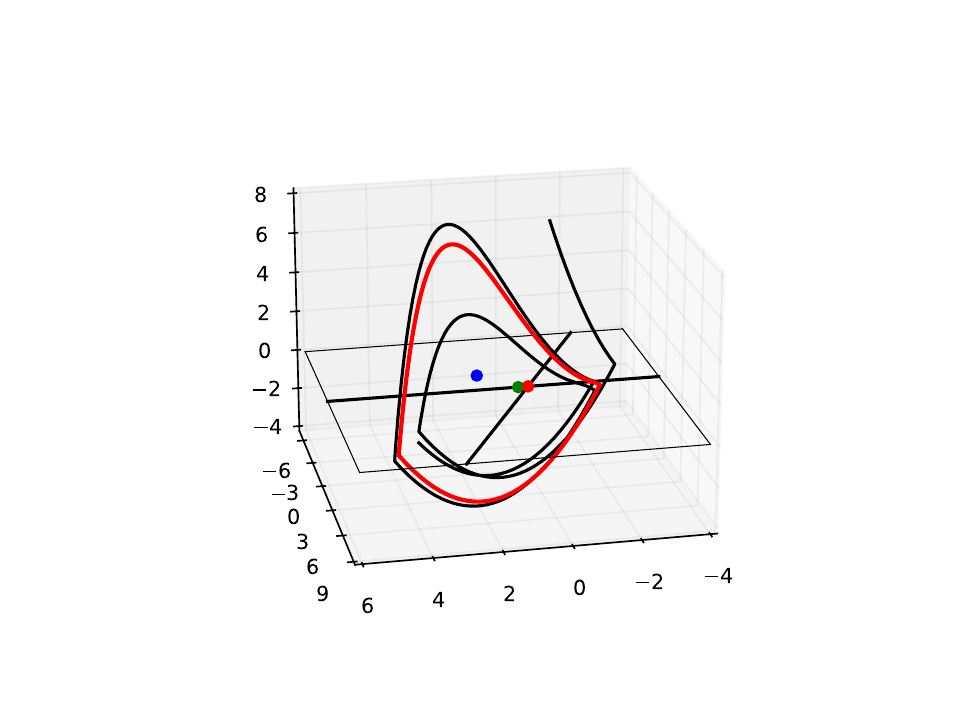}\label{Fig6-Ex1-CuspFold}
				\put(49.2,58){\color{black}\thicklines\vector(-1,1){2}}
				\put(50,55){\color{red}\thicklines\vector(-1,1){2}}
				\put(63,6){\small$x$}
				\put(15,20){\small$y$}
				\put(15,45){\small$z$}
		\end{overpic}}
		\caption{Phase portraits of system \eqref{Ex1} with $b=-1/3$, $\beta=1$ and for some values of $\alpha$ in the range $[-0.3,1.5]$. The red, blue and green dots indicate a T-singularity, a pseudo-equilibrium point and a regular cusp-singularity, respectively. In (b) the pseudo-equilibrium collides with the T-singularity and in (e) the regular cusp singularity collides with the T-singularity giving rise to a cusp-fold singularity. The closed orbit represented in red is a stable CLC in (c), an unstable CLC in (d), a polycycle in (e), and a limit cycle with sliding segment in (f). }\label{Figs-Ex1-CuspFold}
	\end{center}
\end{figure}

To conclude our analysis, we present in Figure \ref{Figs-Ex1-CuspFold} some simulation results where we can visualize different phase portraits of system \eqref{Ex1} based on the change in the value of the $\alpha$ parameter, setting $b=-\dfrac{1}{3}$ and $\beta=1$. We start with $\alpha=1.5$ (we are in the region above the blue line in the bifurcation set shown in Figure \ref{Fig2-Ex1-Setbifurcation2}, that is, before the occurrence of a bifurcation at the T-singularity) and, in this case, the system does not present any CLC and the trajectory tends to a stable pseudo-equilibrium located in the sliding region, as shown in Figure \ref{Fig1-Ex1-CuspFold}. In Figure \ref{Fig2-Ex1-CuspFold}, we take $\alpha=\beta=1$ (we are on the blue line in the bifurcation set shown in Figure \ref{Fig2-Ex1-Setbifurcation2}, that is, at the moment when a bifurcation at the T-singularity occurs) and, in this case, there is also no CLC and the trajectory tends to the T-singularity which now has the property that the vector fields on both sides of the $\Sigma$ switching boundary are anticollinear (this occurs due to the collision of the pseudo-equilibrium with the T-singularity). Next, for $\alpha=0.85$ (we are in the region between the blue and purple lines in the bifurcation set shown in Figure \ref{Fig2-Ex1-Setbifurcation2}) the system has a stable CLC as shown in Figure \ref{Fig3-Ex1-CuspFold}. Continuing in Figure \ref{Fig4-Ex1-CuspFold}, for $\alpha =0.6$ (we are in the region between the red and purple lines in the bifurcation set shown in Figure \ref{Fig2-Ex1-Setbifurcation2}) the CLC still exists but now it is unstable. This unstable CLC becomes a polycycle for $\alpha=0$ (we are on the red line in the bifurcation set shown in Figure \ref{Fig2-Ex1-Setbifurcation2}) and is shown in Figure \ref{Fig5-Ex1-CuspFold}. Note that, simultaneously with the emergence of the polycycle, the double tangency singularity becomes a cusp-fold (the regular cusp-singularity, green dot, collides with the two-fold singularity, red dot). This is just a coincidence that occurs due to the values taken for the $\beta$ and $b$ parameters. Finally, in Figure \ref{Fig6-Ex1-CuspFold} we observe that for $\alpha=-0.3$ (we are in the region below the red line in the bifurcation set shown in Figure \ref{Fig2-Ex1-Setbifurcation2}) the system presents a limit cycle with sliding segment.

\subsection{CLC and polycycle in a power electronic system}\label{Sec-Boost}

We take the closed-loop model of a Boost converter with equations normalized and under a sliding mode control strategy (see \cite{TSboost}), described by the piecewise linear system 
\begin{equation}
	\dot{\mathbf{x}}=\left\{\begin{matrix}
		A^-\mathbf{x}+\mathbf{b},\;\;\text{if}\;\;z<0\\
		A^+\mathbf{x}+\mathbf{b},\;\;\text{if}\;\;z>0\\
	\end{matrix}\right.,\label{sistema_boost2_TSboost}
\end{equation}    
with $$A^-=\begin{bmatrix}
	-b & 0 & 0\\
	0 & -a & 0\\
	-kb & \omega - a & -\omega
\end{bmatrix},\; A^+=\begin{bmatrix}
	-b & -1 & 0\\
	1 & -a & 0\\
	1-kb & \omega - a -k & -\omega
\end{bmatrix},$$
$$\mathbf{x}=\begin{bmatrix}
	x\\
	y\\
	z
\end{bmatrix} \;\text{and}\;\; \mathbf{b}=\begin{bmatrix}
	1\\
	0\\
	k -\omega y_r
\end{bmatrix}     ,$$
where  $(x,y,z)\in 
\mathbb{R}^3$ are the state variables such that $x>0$ is the normalized inductor current, $y> 0$ is the normalized capacitor voltage and $z\in\mathbb{R}$ is an external control variable.
The parameters are $\omega\in(0,1]$, $y_r>1$, $b>0$, $a>0$ and $k>0$. 

In \cite{TSboost} it was shown numerically that the system above presents a CLC that bifurcates from the T-singularity and disappears when it becomes a polycycle.  The CLC and the polycycle are also obtained by us from the simulation of system \eqref{sistema_boost2_TSboost} and shown in Figures \ref{Fig-CLC-boost} and \ref{Fig-Poli-boost}, respectively. 
The bifurcation set in the $(k,a)$-plane of parameters obtained in \cite{TSboost} does not contain the bifurcation curve branch associated with the emergence of the polycycle. 
The objective in what follows is to find such a curve branch and also the region in the $(k,a)$-plane which ensures the existence of a CLC.

\begin{figure}[!h]
	\begin{center}
		\subfigure[$a=1.3$, $k=6$]{
			\begin{overpic}[scale=0.65,clip,trim=2cm 1cm 2cm 2cm]{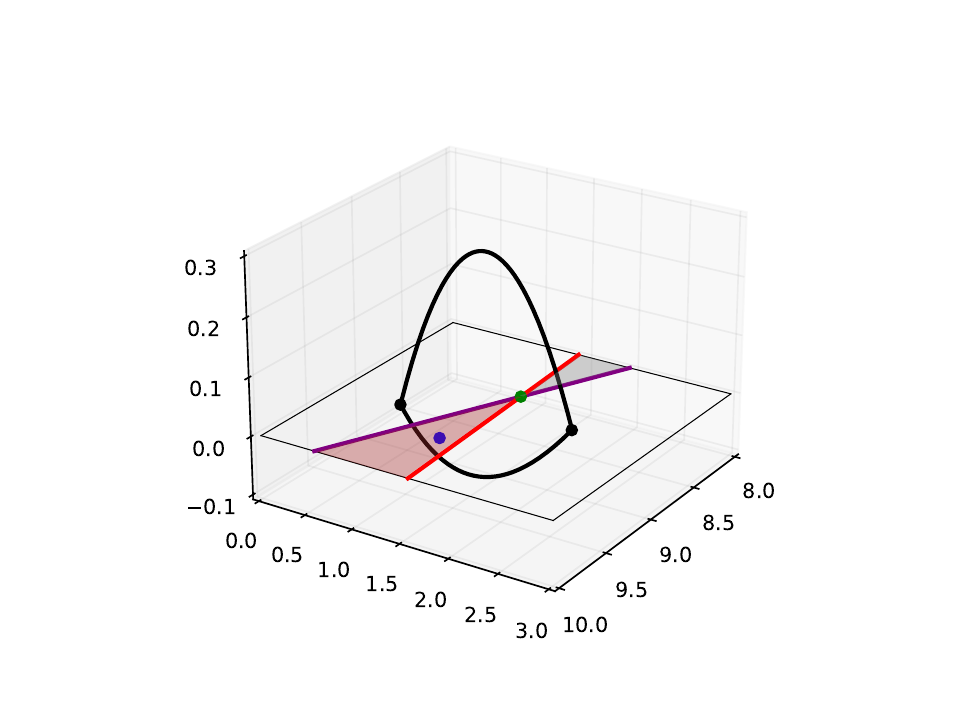}\label{Fig-CLC-boost}
				\put(50.3,25.5){\color{black}\thicklines\vector(-1,0){1}}
				\put(50.9,56.5){\color{black}\thicklines\vector(1,0){1}}
				\put(80,10){$x$}
				\put(28,8){$y$}
				\put(5,40){$z$}
				\put(64,30){\small$\mathbf{x}_1$}
				\put(33,35){\small$\mathbf{x}_0$}
			\end{overpic}
		}
		\subfigure[$a=0.72$, $k=6$]{
			\begin{overpic}[scale=0.65,clip,trim=2cm 1cm 2cm 2cm]{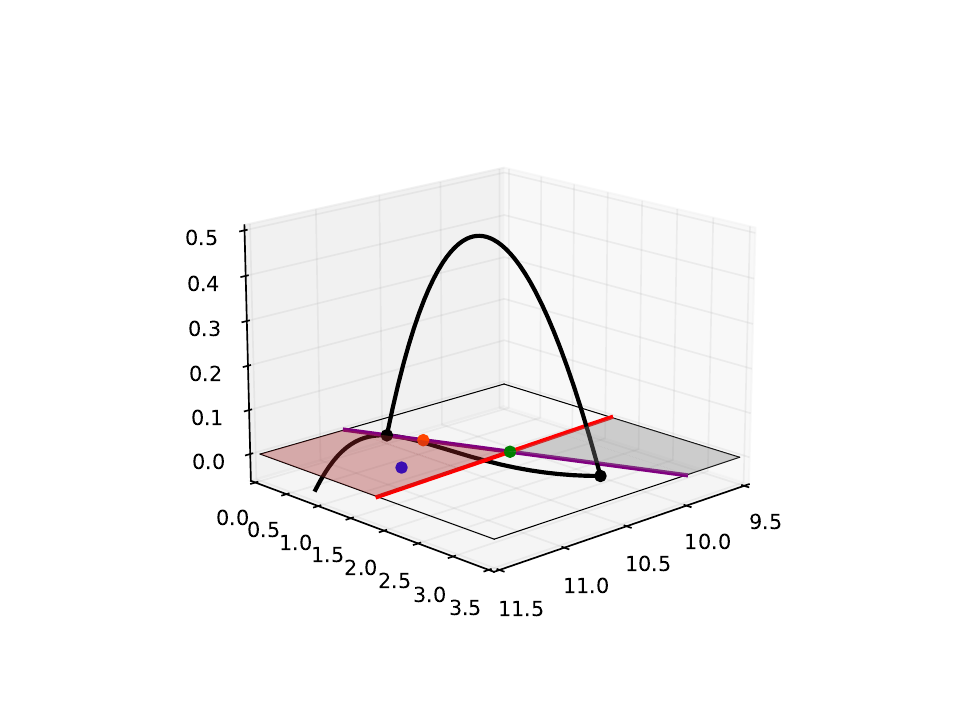}\label{Fig-Poli-boost}
				\put(59,25.9){\color{black}\thicklines\vector(-1,0){1}}
				\put(47.3,57.8){\color{black}\thicklines\vector(1,1){1}}
				\put(80,10){$x$}
				\put(25,10){$y$}
				\put(5,40){$z$}
				\put(65,22){\small$\mathbf{x}_1$}
				\put(39,33){\small$\mathbf{x}_0$}
		\end{overpic}}
		\caption{Phase portraits of system \eqref{sistema_boost2_TSboost} showing the CLC in (a) and the policycle in (b). The blue, green and orange dots indicate a pseudo-equilibrium, T-singularity and Cusp singularity, respectively. The black dots $\mathbf{x}_0\in\Sigma^c_+$ in (a) and $\mathbf{x}_0\in S_Y$ in (b), while $\mathbf{x}_1\in\Sigma^c_-$ in (a) and (b), are the points of intersection of the closed trajectory with the switching boundary $\Sigma=\{z=0\}$. }
	\end{center}
\end{figure}
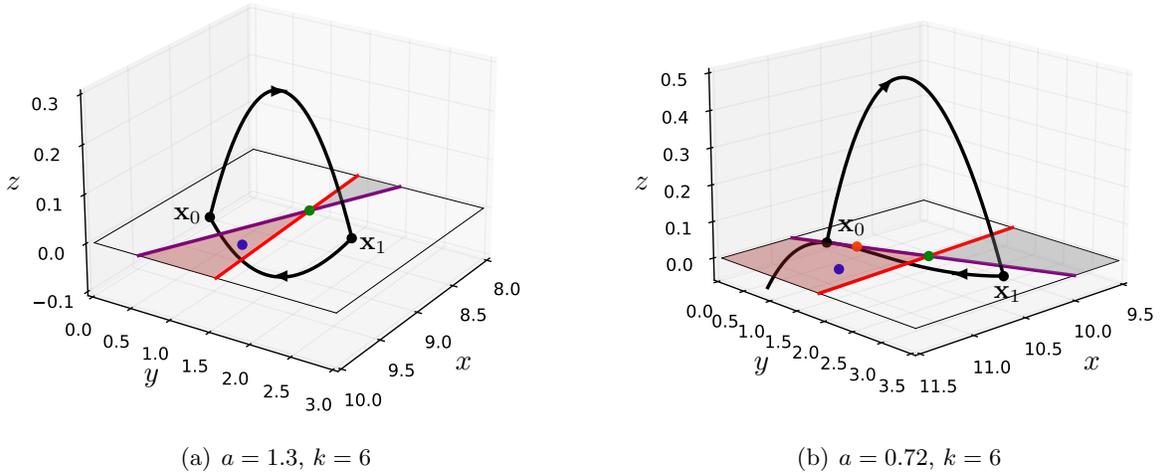

The curve that contains the points $(k,a)$ for which system \eqref{sistema_boost2_TSboost} presents a polycycle is obtained by solving numerically the following system of equations
\begin{align}
	\mathbf{x}_1&=\overline{\mathbf{x}}^{+}+e^{A^{+}\tau_+}(\mathbf{x}_0-\overline{\mathbf{x}}^{+}),\label{eq_fechamento1_TSboost}\\
	\mathbf{x}_0&=\overline{\mathbf{x}}^{-}+e^{A^{-}\tau_-}(\mathbf{x}_1-\overline{\mathbf{x}}^{-}),\label{eq_fechamento2_TSboost}\\
	0&= -kbx_0+(\omega-a)y_0+k-\omega y_r, \label{eq_fechamento3_TSboost}
\end{align}  
where $\mathbf{x}_0=(x_0,y_0,0)\in\Sigma^c_+\cup S_Y$ and $\mathbf{x}_1=(x_1,y_1,0)\in\Sigma_-^c$ are the points of intersection of the CLC with the $z=0$ plane, $\overline{\mathbf{x}}^{-}=\left(\frac{1}{b} ,0 ,-y_r\right)$ and $\overline{\mathbf{x}}^{+}=\left(\frac{a}{1+ab}, \frac{1}{1+ab}, \frac{1}{1+ab}-y_r\right)$  are the equilibria of \eqref{sistema_boost2_TSboost} for $z<0$ and $z>0$, respectively. 
Equations in \eqref{eq_fechamento1_TSboost}-\eqref{eq_fechamento2_TSboost} are the \textit{closing equations} and equation \eqref{eq_fechamento3_TSboost} imposes the condition that $(x_0,y_0,0)\in S_Y$. We consider the unknowns $(x_0,y_0,x_1,y_1,\tau_+,\tau_-,a,k)$ with constants $\omega=0.6$, $y_r=1.33$ and $b=0.08$.


Following \cite{TSboost}, a bifurcation at the T-singularity occurs for $a=a_{TS}(k)$ if $k_1<k\neq\frac{625}{133}<k_2 $, where $k_{1,2}=\frac{625}{133}\left(1\pm\sqrt{1-\frac{300713}{781250}}\right)$ and
$$a_{TS}(k)=k\left(\frac{100}{133}-\frac{2}{25}k\right).$$
The graph of $a=a_{TS}(k)$ is drawn on the parameter plane of Figure \ref{Fig-Set-Bif}, curve in blue color. This bifurcation gives rise to a CLC that appears for $a<a_{TS}(k)$.
We take some sufficient values of $k$ in the range $(k_1,k_2)$ and for each of these values we find the approximate solution of \eqref{eq_fechamento1_TSboost}-\eqref{eq_fechamento3_TSboost} for $(x_0,y_0,x_1,y_1,\tau_+,\tau_-,a)$. The result obtained is a bifurcation branch that contains the points $(k,a)$ for which system \eqref{sistema_boost2_TSboost} presents a polycycle, being represented in the $(k,a)$-plane of Figure \ref{Fig-Set-Bif} in red color.

\begin{figure}
	\begin{center}
		\begin{overpic}[scale=0.7,clip,trim=0cm 0cm 0cm 0cm]{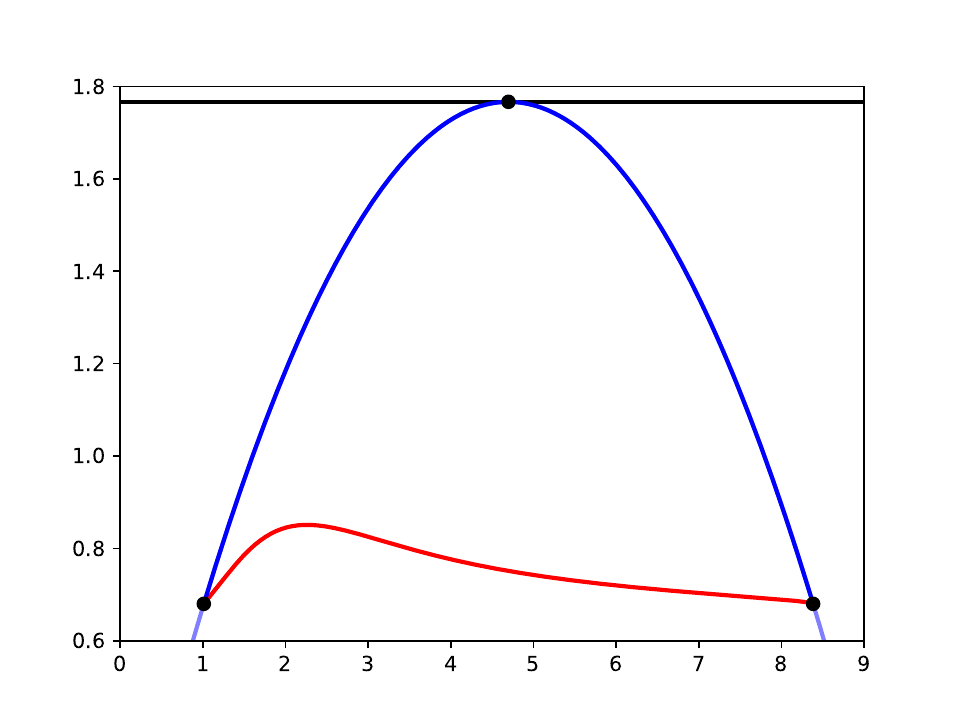}
			\put(40,37){\includegraphics[scale=0.4,clip,trim=4.5cm 3.7cm 4cm 4cm]{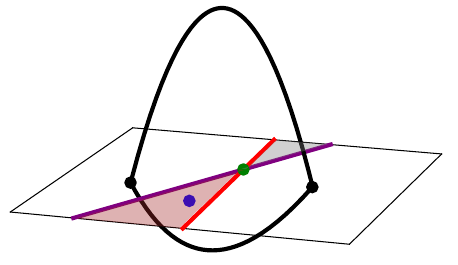}}
			\put(50,20){\includegraphics[scale=0.35,clip,trim=4cm 3cm 4cm 4cm]{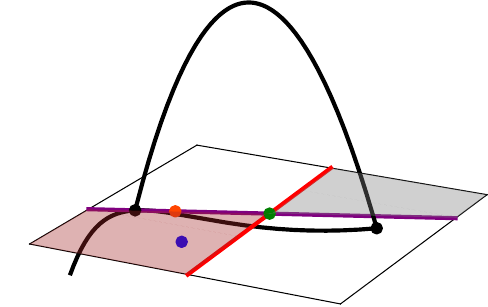}}
			\put(75,45){\includegraphics[scale=0.35,clip,trim=4cm 3cm 4cm 3.5cm]{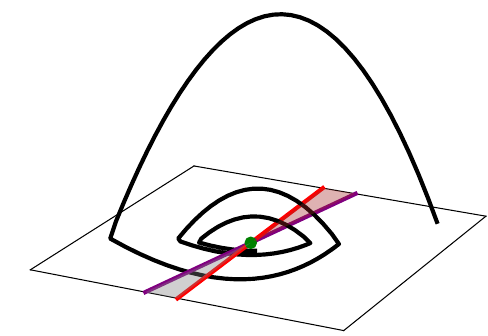}}
			\put(6.5,45){\includegraphics[scale=0.35,clip,trim=4cm 3cm 4cm 3.5cm]{fig-Policiclo-boost6a.pdf}}
			\put(17.7,12){\large$B$}
			\put(86,11){\large$C$}
			\put(51,61){\large$A$}
			\put(50,0){\large$k$}
			\put(0,40){\large$a$}
			\put(60,20){\color{black}\thicklines\vector(-1,-1){5}}
			\put(80,45){\color{black}\thicklines\vector(-1,-1){5}}
			\put(26.5,45){\color{black}\thicklines\vector(1,-1){5}}
		\end{overpic}
		\caption{Bifurcation set in the $(k,a)$-plane of parameters. The blue and red curves indicate the bifurcation of the T-singularity and the appearance of a polycycle, respectively. The black horizontal line indicates a Saddle-Node bifurcation of pseudo-equilibria. At the points $A$, $B$ and $C$ codimension-two bifurcations occur. }\label{Fig-Set-Bif}
	\end{center}
\end{figure}

Based on the bifurcation set shown in Figure \ref{Fig-Set-Bif} we conclude that, for points $(k,a)$ in the region limited by the blue and red curves the system exhibits a CLC, outside this region it does not. The blue and red curves are branches of bifurcations and indicate the occurrence of a T-singularity bifurcation (associated with the birth of a CLC from the T-singularity) and the emergence of a polycycle (associated with the disappearance of a CLC), respectively. At point $A$ we have the simultaneous occurrence of the T-singularity and Saddle-Node (of pseudo-equilibria) bifurcations. At points $B$ and $C$ the double tangency is a cusp-fold singularity and the vector fields above and below the switching boundary are anticollinear at the cusp-fold singularity. From point $B$ (and $C$) the bifurcation branches of the T-singularity and the polycycle emanate. With this, we have characterized the existence and location of the CLC that appears in the model of a Boost converter under a SMC control, from its birth to its vanishing.

\section{Proofs of the main results}

\subsection{Proof of Theorem \ref{NormalFormThm}} \label{normalformproof}

	From the structural stability of a cusp singularity, there exist neighborhoods $U\subset\rn{3}$ of $p$ and $V\subset\rn{k}$ of $0$ such that $X_\delta$ has a cusp at a point $q_\delta\in\s\cap U$ of same type of the cusp of $X_0$ at $p$. For each $\delta\in V$, from Vishik's Theorem (see Theorem \ref{Vishik} and \cite{V}), there is a $C^\infty$ change of coordinates $\psi^\delta$ at a neighborhood $U_\delta$ around $q_\delta$ which brings $q_\delta$ to the origin and such that
	\begin{equation}\label{cusp1}
		\psi^\delta_*X_\delta(x,y,z)=(\mu_1,\mu_2 x,\mu_3 y),
	\end{equation}
	where $|\mu_1|=|\mu_2|=|\mu_3|=1$ and $\psi^\delta(\s)\cap U_\delta=\{(x,y,z);z=0\}$.
	
	\begin{remark}
		In \cite{V}, Vishik has developed a change of coordinates for vector fields defined on a manifold with boundary, and thus the normal form presented by him for this case is the vector field $(1,x,y)$. Nevertheless, in the Filippov systems context there are some crucial differences on the side of the boundary where the vector field is acting. So, we add the signals $\mu_1,\mu_2,\mu_3$ in order to consider all the possible options which are obtained in the Vishik's process if the location of the vector field is considered.
	\end{remark}
	
	Since the methods used on \cite{V} consists mainly on the application of Malgrange's Preparation Theorem, it follows that the change of coordinates $\psi^\delta$ is of class $C^\infty$ on $\delta$, since the dependence of $X_\delta$ on $\delta$ is of class $C^\infty$. So, we can also reduce $U$ and $V$ in such a way that $U_\delta$ can be taken as $U$. 
	
	In order to simplify our future computations, we consider a slight modification of \eqref{cusp1} applying the change of coordinates $(x,y,z)\rightarrow (x-a(\delta),y,z)$, where $a(\delta)$ is a $C^\infty$ function such that $a(0)=0$. Abusing notation, we have that
	\begin{equation}\label{cusp2}
		\psi^\delta_*X_\delta(x,y,z)=(\mu_1,\mu_2 (x+a(\delta)),\mu_3 y).
	\end{equation}
Dividing \eqref{cusp2} by $-\mu_2$ and applying the change of variables $(\overline{x},\overline{y},\overline{z})=\left( \dfrac{\mu_2}{\mu_1}x,\dfrac{\mu_2}{\mu_1}y, \dfrac{-(\mu_2)^2}{\mu_1\mu_3}z \right)$, and rescaling $\overline{a}(\dg)=\dfrac{\mu_2}{\mu_1}a(\dg)$, we obtain that
	\begin{equation}\label{cusp3}
	\psi^\delta_*X_\delta(x,y,z)=(-1,- (x+a(\delta)), y).
\end{equation}
		After this sequence of change of coordinates, we have that 
	$$\psi^\delta_*Y_\delta=(Y_{\delta,1},Y_{\delta,2},Y_{\delta,3}),$$
	where $Y_{0,3}(0,0,0)=0$ and $\partial_xY_{0,3}(0,0,0)Y_{0,1}(0,0,0)+\partial_yY_{0,3}(0,0,0)Y_{0,2}(0,0,0)\neq 0$.
	
	\begin{remark}
		We keep the same notation between the variables $(x,y,z)$, $a$ and $(\overline{x},\overline{y},\overline{z})$, $\overline{a}$ in order to simplify the expressions.
	\end{remark}
	
	Notice that if $f(x,y,z)=z$, then $(\psi^\delta_*X_\delta) f(x,y,0)=0$ if and only if $y=0$, so it follows that $\psi^\delta_*X_\delta$ has a tangency line with $\s$ at the $x$-axis which is composed by fold points and a cusp point at $(-a(\delta),0,0)$. On the other hand, $\psi^\delta_*Y_\delta$ has a tangency line at the set $Y_{\delta,3}(x,y,0)=0$ composed by fold points, and since the contact of the tangency lines of $X_0$ and $Y_0$ is transverse at $p$ it follows that $ (\partial_xY_{0,3}(0,0,0),\partial_yY_{0,3}(0,0,0))$ and $(0,1)$ are linearly independent, which means that $\partial_xY_{0,3}(0,0,0)\neq 0$. Hence, it follows from the Implicit Function Theorem that there exists a $C^\infty$ function $b_\delta(y)$ such that the tangency line of $\psi^\delta_*Y_\delta$ is given by $x=b_\delta(y)$, where $b_0(0)=0$. 
	
	Thus the double-singularity of $\psi^\delta_*Z_{\dg}$ occurs at $(b_\delta(0),0,0)$, and applying the change of variables 
	$(x,y,z)\rightarrow (x-f_\delta,y,z)$, where $f_\delta=b_\delta(0)=\mathcal{O}(\delta)$, and choosing $a(\delta)=c_\delta-f_\delta$, we  obtain that this system is brought to (we abuse the notation by mantaining the same name of the vector field)
	\begin{equation}
		\psi^\delta_*Z_\delta=\left\{\begin{array}{lcl}
		(-1,- (x+c_\delta), y), &\textrm{ if}& z>0,  \\
			(Y_{\delta,1}(x+f_\delta,y,z),Y_{\delta,2}(x+f_\delta,y,z),Y_{\delta,3}(x+f_\delta,y,z)),& \textrm{ if}&z<0.
		\end{array}\right.
	\end{equation}
		In this case, $\psi^\delta_*Z_\delta$ has a double-tangency singularity at the origin and the cusp-singularity of $\psi^\delta_*X_\delta$ is located at $(-c_\delta, 0,0)$. Now, expanding $\psi^\delta_*Y_\delta$ in the variables $(x,y,z)$ around the origin, we have that
	$$\psi^\delta_*Y_\delta=\left(\displaystyle\sum_{0\leq i,j,k\leq n}\ag_{i,j,k}(\delta)x^iy^jz^k, \displaystyle\sum_{0\leq i,j,k\leq n}\bg_{i,j,k}(\delta)x^iy^jz^k, \displaystyle\sum_{0\leq i,j,k\leq n}\cg_{i,j,k}(\delta)x^iy^jz^k\right)+h.o.t.$$
	
	Since $Y_{\delta,3}(f_\delta,0,0)=Y_{\delta,3}(b_\delta(0),0,0)=0$, if follows that $\cg_{0,0,0}(\delta)\equiv0$. Also $c_0=0$ and $\ag_{0,0,0}(\dg)\cg_{1,0,0}(\dg)+\bg_{0,0,0}(\dg)\cg_{0,1,0}(\dg)\neq 0$ has constant sign.
	
	This completes the proof of Theorem \ref{NormalFormThm}.

\subsection{Proof of Theorem \ref{thmnonexistence}} \label{nonexistenceproof}

The proof of this Theorem is divided into three step. In step $1$ we construct half-return maps for the unfolding given in \eqref{normalformeq} when the fold singularity is invisible. In step $2$, based on the half-return maps of the previous step, we construct a displacement map that can be used to analyze the one-loop closed orbits of the vector field. Finally, in step $3$ we analyze the unfolding of a vector field around a cusp-fold singularity of degree $1$ and we conclude that for small perturbations, the displacement map has only one zero which corresponds to the double-tangency singularity (which can be a cusp-fold or a two-fold singularity).

\subsubsection{Step 1: Half-return maps} 

Let $Z_0$ be a Filippov system having an invisible cusp-fold singularity at $p\in\s$ in such a way that $p$ is a cusp of $X_0$ and an invisible fold of $Y_0$. Consider a $k$-dimensional unfolding $Z_\delta=(X_\delta,Y_\delta)$ of $Z_0$ at $p$.

From Theorem \ref{NormalFormThm}, there is no loss of generality in assuming that $Z_\delta$ is given by \eqref{normalformeq}. 
From the invisibility of the fold of $Y_0$ we have that $\ag_{0,0,0}(\dg)\cg_{1,0,0}(\dg)+\bg_{0,0,0}(\dg)\cg_{0,1,0}(\dg)>0$.  
Also, from now on,  we omit the $\delta$-dependence of the coefficients $\ag_{i,j,k}(\dg),\bg_{i,j,k}(\dg),\cg_{i,j,k}(\dg)$ and $c(\dg)$ in order to provide a cleaner version of the proof. So, this $k$-dimensional unfolding of $Z_0$ is simply given by
\begin{equation}\label{normalformeq2}
	Z=(X,Y)=\left\{\begin{array}{lcl}
		(-1,-(x+c), y), & \textrm{if}& z>0,  \\
		{\small \left(\displaystyle\sum\ag_{i,j,k}x^iy^jz^k, \displaystyle\sum\bg_{i,j,k}x^iy^jz^k, \displaystyle\sum\cg_{i,j,k}x^iy^jz^k\right)},& \textrm{if}&z<0.\end{array}\right.
\end{equation}
Recall that $\cg_{0,0,0}\equiv0$. Now we construct half-return maps in order to identify one-loop CLCs.

\begin{lemma}\label{cuspreturnlem}
	The half-return map $\Pi_X:\s\rightarrow\s$ associated to the first component of the Filippov system \eqref{normalformeq2} is given by
	\begin{equation}\label{halfreturncusp}
		\Pi_X(x,y)=\left( \begin{array}{c}
			x-\dfrac{1}{2}\left(3(x+c)-\sqrt{9(x+c)^2-24y}\right)\vspace{0.2cm} \\
			\dfrac{1}{4}(x+c)\left(3(x+c)-\sqrt{9(x+c)^2-24y}\right)-2y
		\end{array}\right).
	\end{equation}
\end{lemma}
\begin{proof}
	Here $X(x,y,z)=(-1,-(x+c), y)$, thus its flow is given explicitly by
	$$\varphi_X(t,(x,y,z))=\left(
	\begin{matrix}
		x-t\vspace{0.2cm}\\
		\dfrac{(x-t)^2}{2}-ct-\dfrac{x^2}{2}+y\vspace{0.2cm}\\
		-\dfrac{(x-t)^3}{6}-\dfrac{ct^2}{2}+\left(\dfrac{x^2}{2}+y\right)t+\dfrac{x^3}{6}+z
	\end{matrix}
	\right).$$
	We obtain the return times of an orbit trough the solutions of the equation $$\pi_3\circ \varphi_X(t,(x,y,0))=0,$$ which are given by $t_0=0$ and
	$$t_{\pm}=\dfrac{3(x+c)\pm\sqrt{9(x+c)^2-24y}}{2}.$$
	Since $t_0$ is the trivial solution and $t_-<t_+$, it follows that the first return time is given by $t=t_-$, and we define $\Pi_X(x,y)=\varphi_X(t_-,(x,y,0))$. Finally using the expressions above, we obtain \eqref{halfreturncusp}.
\end{proof}

\begin{lemma}\label{foldreturnlem}
	The half-return map $\Pi_Y:\s\rightarrow\s$ associated to the second component of the Filippov system \eqref{normalformeq2} is given by
	\begin{equation}\label{halfreturncusp}
		\Pi_Y(x,y)=\left( \begin{array}{c}
			P_1(x,y)\vspace{0.2cm}\\
			P_2(x,y)
		\end{array}\right),   
	\end{equation}
	where $P_1(0,0)=P_2(0,0)=0$ and $\Pi_Y^2(x,y)=(x,y)$. Furthermore,
	\begin{equation}\label{foldreturnapp}
		\Pi_Y(x,y)=\left(\begin{matrix}
			1-\dfrac{2\cg_{1,0,0}\ag_{0,0,0}}{\cg_{1,0,0}\ag_{0,0,0}+\cg_{0,1,0}\bg_{0,0,0}}& \dfrac{-2\cg_{0,1,0}\ag_{0,0,0}}{\cg_{1,0,0}\ag_{0,0,0}+\cg_{0,1,0}\bg_{0,0,0}}\vspace{0.2cm}\\
			\dfrac{-2\cg_{1,0,0}\bg_{0,0,0}}{\cg_{1,0,0}\ag_{0,0,0}+\cg_{0,1,0}\bg_{0,0,0}}& 1-\dfrac{2\cg_{0,1,0}\bg_{0,0,0}}{\cg_{1,0,0}\ag_{0,0,0}+\cg_{0,1,0}\bg_{0,0,0}}
		\end{matrix}\right)\left(\begin{matrix}
			x \\
			y
		\end{matrix}\right) +\er(|(x,y)|^2).
	\end{equation}
\end{lemma}
\begin{proof}
	Let $\varphi_Y(t,(x,y,z))$ denote the flow of the vector field $Y$ given in \eqref{normalformeq2} and notice that $\varphi_Y(0,(x,y,z))=(x,y,z)$,$$\partial_t\varphi_Y(0,(x,y,z))=Y(\varphi_Y(t,(x,y,z))|_{t=0}=Y(x,y,z),$$
	and
	$$\partial_t^2\varphi_Y(0,(x,y,z))=DY(\varphi_Y(t,(x,y,z))\cdot\partial_t\varphi_Y(t,(x,y,z)|_{t=0}=DY(x,y,z)\cdot Y(x,y,z).$$
	Thus, it follows that
	
	$$\varphi_Y(t,(x,y,z))=(x,y,z)+ tY(x,y,z)+\dfrac{t^2}{2}DY(x,y,z)\cdot Y(x,y,z)+\er(t^3).$$
	
	Given $(x,y,0)\in\s$, we have that $\varphi_Y(t,(x,y,0))\in\s$ if and only if $\pi_3\circ \varphi_Y(t,(x,y,0))=0,$
	which gives us the equation
	$$\widetilde{R}(x,y,t):=tY_3(x,y,0)+\dfrac{t^2}{2}\left(\partial_xY_3Y_1+\partial_yY_3Y_2+\partial_zY_3Y_3\right)|_{(x,y,0)}+\er(t^3)=0,$$
	where $Y=(Y_1,Y_2,Y_3)$. Considering $R(x,y,t)=\widetilde{R}(x,y,t)/t$, we have that $R(0,0,0)=\cg_{0,0,0}=0$ and $$\dfrac{\partial R}{\partial t}(0,0,0)=\dfrac{1}{2}\left(\cg_{1,0,0}\ag_{0,0,0}+\cg_{0,1,0}\bg_{0,0,0}\right)>0.$$
	From the Implicit Function Theorem, it follows that there exists a $\Cr$-function $\tau:B(0,0)\rightarrow \mathbb{R}$, where $B(0,0)$ is a small ball centered at the origin, such that $\tau(0,0)=0$ and $R(t,x,y)=0$ with $(x,y)\in B(0,0)$ if and only if $t=\tau(x,y)$. Also, it follows that
	$$\partial_x\tau(0,0)=\dfrac{-2\cg_{1,0,0}}{\cg_{1,0,0}\ag_{0,0,0}+\cg_{0,1,0}\bg_{0,0,0}} \textrm{ and }\partial_y\tau(0,0)=\dfrac{-2\cg_{0,1,0}}{\cg_{1,0,0}\ag_{0,0,0}+\cg_{0,1,0}\bg_{0,0,0}}.$$
	
	Hence, we define $\Pi_Y:B(0,0)\times\{0\}\rightarrow\s$ by $\Pi_Y(x,y,0)=\varphi_Y(\tau(x,y),(x,y,0))$, and using the expressions above we obtain equation \eqref{foldreturnapp}. Notice that since $t=0$ and $t=\tau(x,y)$ are the only return times of a point $(x,y,0)\in\s$, it follows that
	$$\begin{array}{lcl}
		\Pi_Y(\Pi_Y(x,y,0))& = &\varphi_Y(\tau(\Pi_Y(x,y,0)),\Pi_Y(x,y,0)) \\
		& = &\varphi_Y(-\tau(x,y),\varphi_Y(\tau(x,y),(x,y,0)))\\
		& = &(x,y,0).
	\end{array}$$
\end{proof}

\subsubsection{Step 2: Displacement map}

\begin{lemma}\label{displamentlem}
	Let $Z$ be the Filippov system given by \eqref{normalformeq2}. The orbit of $Z$ through $(x,y,0)\in\s$ is closed and intersects $\s$ at only two points if and only if $9(x+c)^2-24y\geq 0$ and $G(x,y,c)=(0,0)$, where 
	
	\begin{equation}\label{gfunc}
		G(x,y,c)=\left(\begin{matrix}
			6y+    (P_1(x,y)-x) (P_1(x,y)+3 c +2 x)\vspace{0.2cm} \\
			6 P_2(x,y)-(P_1(x,y)-x) (2 P_1(x,y)+3 c +x)
		\end{matrix}\right),
	\end{equation}    
	where $\Pi_Y=(P_1,P_2)$, is the half-return map given by Lemma \ref{foldreturnlem}. Moreover, if $G(x,y,c)=0$ then $G(P_1(x,y),P_2(x,y),c)=0$.
\end{lemma}
\begin{proof}
	In fact, the orbit of $Z$ through $(x,y,0)\in\s$ is closed and intersects $\s$ at only two points if and only if $\Pi_X(x,y,0)=\Pi_Y(x,y,0)$, where $\Pi_X,\Pi_Y$ are the half-return maps given by \ref{cuspreturnlem} and \ref{foldreturnlem}, respectively. Thus
	$$
	\begin{array}{lcl}
		\Pi_X=\Pi_Y & \Longleftrightarrow & \left\{\begin{array}{l}
			x-\dfrac{1}{2}\left(3(x+c)-\sqrt{9(x+c)^2-24y}\right)=P_1(x,y)\vspace{0.2cm} \\
			\dfrac{1}{4}(x+c)\left(3(x+c)-\sqrt{9(x+c)^2-24y}\right)-2y=P_2(x,y)
		\end{array} \right., \vspace{0.4cm}\\
		& \Longleftrightarrow & \left\{\begin{array}{l}
			2(x-P_1(x,y))= 3(x+c)-\sqrt{9(x+c)^2-24y}\vspace{0.2cm} \\
			4P_2(x,y)+8y=(x+c)\left(3(x+c)-\sqrt{9(x+c)^2-24y}\right)
		\end{array} \right., \vspace{0.4cm}\\
		& \Longleftrightarrow & \left\{\begin{array}{l}
			\sqrt{9(x+c)^2-24y} = x+3c+2P_1(x,y)\vspace{0.2cm} \\
			4P_2(x,y)+8y=2(x+c)(x-P_1(x,y))
		\end{array} \right. ,\vspace{0.4cm}\\     
		& \Longleftrightarrow & \left\{\begin{array}{l}
			9(x+c)^2-24y -(x+3c+2P_1(x,y))^2=0\vspace{0.2cm} \\
			4P_2(x,y)+8y-2(x+c)(x-P_1(x,y))=0
		\end{array} \right., \vspace{0.2cm}\\ 
		& \Longleftrightarrow & \left\{\begin{array}{l}
			6y+    (P_1(x,y)-x) (P_1(x,y )+3 c +2 x)=0\vspace{0.2cm} \\
			6 P_2(x,y)-(P_1(x,y)-x) (2 P_1(x,y)+3 c +x)=0
		\end{array} \right.. \vspace{0.2cm}\\  
	\end{array}
	$$
	The result follows directly.
	Finally, if $G(x,y,c)=0$ then the orbit through $(x,y,0)$ is closed and it passes through $(P_1(x,y),P_2(x,y),0)$. Hence, the orbit through $(P_1(x,y),P_2(x,y),0)$ is also closed, which means that $G(P_1(x,y),P_2(x,y),c)=0$.
\end{proof}

\begin{remark}\label{remdisp}
	Notice that, from the geometry of the problem, if $(x,y,0)$ satisfies $G(x,y,0)$ and $9(x+c)^2-24y\geq 0$, then $(P_1(x,y),P_2(x,y),0)$ must satisfy that $9(P_1(x,y)+c)^2-24P_2(x,y)\geq 0$.
\end{remark}

\subsubsection{Final Step: Zeroes of the displacement map}\label{zeroesthm3sec}


Finally, we are able to prove Theorem \ref{thmnonexistence}. In fact, from Theorem \ref{NormalFormThm}, there is no loss of generality in assuming that $p_0$ is the origin and  for a small neighborhood $\mathcal{U}\subset\Or$ of $Z_0$, each $Z\in\mathcal{U}$ can be written as \eqref{normalformeq2}. Moreover, since Lemma \ref{displamentlem} provides a displacement map $G=(G_1,G_2)$  for this normal form  \eqref{normalformeq2}, it is sufficient to analyze their zeroes.
	
	Since the origin is a double tangency of $Z$, it follows that $G(0,0,c)\equiv 0$, also, we can see that $G_1(0,0,0)=0$ and $\partial_yG_1(0,0,0)=6$. From the Implicit Function Theorem, there exists a function $\textbf{y}:U\rightarrow \mathbb{R}$, where $U$ is a small neighborhood of $0\in\rn{2}$, such that $\textbf{y}(0,0)=0$ and $G_1(x,y,c)=0$ with $(x,c)\in U$ if and only if $y=\textbf{y}(x,c)$. Since $\partial_xG_1(0,0,0)=\partial_cG_1(0,0,0)=0$, it follows that $\textbf{y}(x,c)=\er(|(x,c)|^2)$.
	
	Now, consider the function $\widetilde{G_2}(x,c)=G_2(x,\textbf{y}(x,c),c)$, and notice that $\widetilde{G_2}(0,0)=0$ and $\partial_x\widetilde{G_2}(0,0)=6\partial_xP_2(0,0,0)$, where $P_2$ is given in Lemma \ref{foldreturnlem}, which gives us that $\partial_x\widetilde{G_2}(0,0)=\dfrac{-12\cg_{1,0,0}\bg_{0,0,0}}{\cg_{1,0,0}\ag_{0,0,0}+\cg_{0,1,0}\bg_{0,0,0}}$.
	
	Notice that $Z_0$ corresponds to the case $c=0$ at the normal form given in \eqref{normalformeq2} which is denoted by $Z_0^N=(X_0^N,Y_0^N)$. Also, since $p_0$ is of degree $1$, we have that $S_{X_0}\pitchfork S_{Y_0}$ at $p_0$. Thus, it follows that the vectors $(0,1)$ and $(\cg_{1,0,0},\cg_{0,1,0})$ are linearly independent (since they represent the normal vector to the tangency line of $X_0^N$ and $Y_0^N$ in the normal form for $c=0$, respectively). Consequently $\cg_{1,0,0}\neq 0$.
	
	Since the nonvanishing of Lie derivatives are preserved through change of coordinates and time-rescaling and $p_0$ is of degree $1$, it follows that $Y_0X_0f(p_0)\neq 0$, and thus  $\bg_{0,0,0}=Y_0^NX_0^Nf(0)\neq 0$. In this case, $\partial_x\widetilde{G_2}(0,0)\neq0$. From the Implicit Function Theorem, there exists a unique function $\textbf{x}:(-\e,\e)\subset\rn{}$ such that $\widetilde{G_2}(x,c)=0$ with $c\in(-\e,\e)$ if and only if $x=\textbf{x}(c)$.
	
	Hence, reducing $U$ and $\mathcal{U}$ if necessary, we can see that, for any $Z\in\mathcal{U}$, the corresponding equation $G(x,y,c)=(0,0)$ has a unique zero $z(c)=(\textbf{x}(c),\textbf{y}(\textbf{x}(c),c))$ in $U$. Since $G(0,0,c)\equiv 0$, it follows that $z(c)\equiv (0,0)$. Thus, $G$ has a unique zero around the origin which corresponds to a $\s$-singularity which can be a cusp-fold or a fold-fold singularity and it is located exactly at the origin.The result follows directly.

\subsection{Proof of Theorem \ref{Teo-CLC-Policiclo}}

Let $Z_\dg=(X_{\dg},Y_{\dg})$ be a $k$-dimensional unfolding of a Filippov system $Z_0=(X_0,Y_0)$ which has an invisible cusp-fold singularity at $p$. From Theorem \ref{NormalFormThm}, there exists a change of coordinates and a rescaling of time which brings $p$ to the origin and $Z_\dg$ onto equation \eqref{normalformeq}. 

For $Z_0$, we have that $c=0$ in \eqref{normalformeq} and thus it is given by equation \eqref{normalformeZ0}. Assume that the origin is an invisible cusp-fold singularity of degree $2$ with index $L_0$ of $Z_0$.

From Definition \ref{degdef}, it follows that $Y_0X_0f(p)= 0$, which corresponds to  $\bg_{0,0,0}=0$ in the normal form \eqref{normalformeq}. 

\begin{remark}
	We omit the dependence of $Z_\dg=(X_\dg,Y_\dg)$ on $\dg$ when it is convenient. So we might write $(X,Y)$ instead of $(X_{\dg},Y_\dg)$. 
\end{remark}

From Lemma \ref{cuspreturnlem}, we have that the half-return map $\Pi_{X}:\s\rightarrow \s$ associated to $X$ is given by
 $$
		\Pi_X(x,y)=\left( \begin{array}{c}
	x-\dfrac{1}{2}\left(3(x+c)-\sqrt{9(x+c)^2-24y}\right)\vspace{0.2cm} \\
	\dfrac{1}{4}(x+c)\left(3(x+c)-\sqrt{9(x+c)^2-24y}\right)-2y
\end{array}\right).
$$
From Lemma \ref{foldreturnlem}, we have also constructed the half-return map $\Pi_Y:\s\rightarrow\s$ associated to $Y$. Nevertheless, in order to analyze this most degenerate case, we need to know a more complete asymptotic expression of $\Pi_Y=(P_1,P_2)$ in terms of the coefficients of the vector field $Y$.

\begin{lemma}\label{piyasymptoticlem}
	Let $\Pi_Y=(P_1,P_2)$ be the half-return map obtained in Lemma \ref{foldreturnlem} associated to the vector field $$Y= {\small \left(\displaystyle\sum\ag_{i,j,k}x^iy^jz^k, \displaystyle\sum\bg_{i,j,k}x^iy^jz^k, \displaystyle\sum\cg_{i,j,k}x^iy^jz^k\right)}.$$
	There exist coefficients $a_{i,j},\ b_{i,j}$, $0\leq i+j\leq 3$, $i,j\geq 0$, which depends on the $\delta$-parameters $\ag_{0, 0, 0}$, $\ag_{0, 0, 1}$, $\ag_{0, 1, 0}$, $\ag_{0, 2, 0}$, $\ag_{1, 0, 0}$, $\ag_{1, 1, 0}$, $\ag_{2, 0, 0}$, $\bg_{0, 0, 0}$, $\bg_{0, 0, 1}$, $\bg_{0, 1, 0}$, $\bg_{0, 2, 0}$, $\bg_{1, 0, 0}$, $\bg_{1, 1, 0}$, $\bg_{2, 0, 0}$, $\cg_{0, 0, 1}$, $\cg_{0, 1, 0}$, $\cg_{0, 1, 1}$, $\cg_{0, 2, 0}$, $\cg_{0, 3, 0}$, $\cg_{1, 0, 0}$, $\cg_{1, 0, 1}$, $\cg_{1, 1, 0}$, $\cg_{1, 2, 0}$, $\cg_{2, 0, 0}$, $\cg_{2, 1, 0}$, $\cg_{3, 0, 0}$, such that
	\begin{equation}\label{halfreturnpiyeq}
		\Pi_Y(x,y)={\small \left(\displaystyle\sum_{i+j=0}^3a_{i,j}x^iy^j, \displaystyle\sum_{i+j=0}^3b_{i,j}x^iy^j\right)}+ \er(|(x,y)|^4).
	\end{equation}
	Moreover, the coefficients $a_{i,j},\ b_{i,j}$, $0\leq i+j\leq 3$, $i,j\geq 0$ can be obtained through the algorithm available in \ref{coeffpiY} (the expressions are omitted due to the extension of the expressions). 
\end{lemma}
\begin{proof}
	In order to obtain the expression \eqref{halfreturnpiyeq}, we expand the flow of $Y$, $\varphi_Y(t;(x,y,0))$ up to order $4$ in the variables $t,x,y$ around the point $t=x=y=0$. Hence, we expand the return time $T(x,y)$ such that $\pi_3\circ \varphi_Y(T(x,y);(x,y,0))=0$ up to order $3$ in the variables $x,y$ around $x=y=0$.
	
	Finally, we consider $\Pi_Y(x,y)=(\pi_1\circ \varphi_Y(T(x,y);(x,y,0)), \pi_2\circ \varphi_Y(T(x,y);(x,y,0))),$ and we expand it up to order $3$.
\end{proof}

From Lemma \ref{displamentlem}, we have that the closed orbits of $Z=(X,Y)$ are detected by the zeroes of 
the map $G(x,y,c)=(G_1(x,y,c),G_2(x,y,c))$ given by \eqref{gfunc}. Now, Lemma \ref{piyasymptoticlem} gives us a more complete knowledge on the expression of $G$.

As we have seen before in Section \ref{zeroesthm3sec}, it follows from the Implicit Function Theorem that there exists a function $\textbf{y}:U\rightarrow \mathbb{R}$, where $U$ is a small neighborhood of $0\in\rn{2}$, such that $\textbf{y}(0,0)=0$ and $G_1(x,y,c)=0$ with $(x,c)\in U$ if and only if 
\begin{equation}\label{eqy}
y=\textbf{y}(x,c).
\end{equation}
Moreover, $\textbf{y}(x,c)=\er(|(x,c)|^2)$. Also, consider the function 
\begin{equation}\label{G2}
	\widetilde{G_2}(x,c)=G_2(x,\textbf{y}(x,c),c),
\end{equation}
and notice that $\widetilde{G_2}(0,c)\equiv 0$ since the origin is a fixed point for every $c$. Using equation \eqref{gfunc} and Lemma \ref{piyasymptoticlem}, we can easily obtain an asymptotic expression for $\widetilde{G_2}$ which is given in the following statement.

\begin{lemma}\label{zerofunclem}
	Let $\widetilde{G_2}$ be the function given by \eqref{G2}. Then $\widetilde{G_2}$ is given by
	$$\widetilde{G_2}(x,c,\bg_{0,0,0})= g_1(c,\bg_{0,0,0})x+ g_2(c,\bg_{0,0,0})x^2+ g_3(c,\bg_{0,0,0})x^3+\er(x^4),$$
	where $g_1,g_2,g_3$ are $\dg$-parameters depending on the coefficients $\ag_{i,j,k},\bg_{i,j,k},\cg_{i,j,k}$, $0\leq i+j+k\leq 3$ of the second component of the vector field \eqref{normalformeq2}. Moreover, $g_1(c,\ag_{0,0,0}c)=g_2(c,\ag_{0,0,0} c)\equiv 0$ and 
		$$\begin{array}{lcl}
		
		g_3(0,0)&=    & -\dfrac{4}{3\ag_{0,0,0}^2\cg_{1,0,0}}\left(-\alpha _{0,0,0} \gamma _{1,0,0} \left(\alpha _{1,0,0}-\beta _{0,1,0}-\beta _{2,0,0}+\gamma _{0,0,1}\right)\right. \vspace{0.3cm} \\
		
		&+  &\left. \beta _{1,0,0} \left(2 \gamma _{0,1,0}+\gamma _{2,0,0}\right)\right)+\alpha _{0,0,0}^2 \left(\gamma _{0,1,0}+\gamma _{2,0,0}\right) \vspace{0.3cm} \\ 
		
		&-& \left.\beta _{0,0,1} \gamma _{1,0,0}^2+\beta _{1,0,0}^2 \gamma _{0,1,0}+\beta _{1,0,0} \gamma _{1,0,0} \left(\gamma _{0,0,1}-\beta _{0,1,0}\right)\right).
	\end{array}$$
	
\end{lemma}

\begin{remark}
	We notice that the coefficients $g_1,g_2,g_3$ are $\dg$-parameters which depends on the coefficients $\ag_{i,j,k},\bg_{i,j,k},\cg_{i,j,k}$, $0\leq i+j+k\leq 3$, $i,j,k\geq 0$  and the parameter $c$. Nevertheless, we explicit the dependence on $\bg_{0,0,0}$ and $c$, since these parameters will play an important role to describe the bifurcations of $Z_{\dg}$.
\end{remark}    

More specifically, one can obtain the following result.

\begin{lemma}
	The functions $g_i$, $i=1,2,3$, given be Lemma \ref{zerofunclem} satisfies that
	$$g_i(c,\bg_{0,0,0})=g_{i,0}(c)+g_{i,1}(c)\bg_{0,0,0}+\er(\bg_{0,0,0}^2),$$
	where
	\begin{enumerate}[i)]
		\item $g_{1,0}(c)=4c+\er(c^2)$;\vspace{0.2cm}
		\item $g_{1,1}(c)= \frac{-4}{\ag_{0,0,0}}+\er(c)$;\vspace{0.2cm}
		\item $g_{2,0}(c)= \frac{-4 \gamma _{1,0,0} \left(\alpha _{1,0,0}+3 \beta _{0,1,0}+\gamma _{0,0,1}\right)+4 \alpha _{0,0,0} \left(\gamma _{2,0,0}-2 \gamma _{0,1,0}\right)+8 \beta _{1,0,0} \gamma _{0,1,0}}{3 \alpha _{0,0,0} \gamma _{1,0,0}}c+\er(c^2)$;\vspace{0.2cm}
		\item $g_{2,1}(c)= \frac{4 \left(2 \gamma _{0,1,0} \left(\alpha _{0,0,0}-\beta _{1,0,0}\right)+\gamma _{1,0,0} \left(\alpha _{1,0,0}+3 \beta _{0,1,0}+\gamma _{0,0,1}\right)-\alpha _{0,0,0} \gamma _{2,0,0}\right)}{3 \alpha _{0,0,0}^2 \gamma _{1,0,0}}+\er(c)$; \vspace{0.2cm}
		\item $g_{3,0}(c)=g_3(0,0)+\er(c)$; \vspace{0.2cm}
		\item $g_{3,1}(c)= \frac{4 P_{3,1}}{9 \alpha _{0,0,0}^3 \gamma _{1,0,0}^2}+\er(c)$ where $P_{3,1}$ is a polynomial of degree $3$ in the variables $\alpha _{0,0,0},\alpha _{0,0,1},\alpha _{0,1,0},$ $\alpha _{1,0,0}, \alpha _{2,0,0},\beta _{0,0,1},\beta _{0,1,0},\beta _{1,0,0},\beta _{1,1,0},\beta _{2,0,0},\gamma _{0,0,1},$ 
	$\gamma _{0,1,0},\gamma _{1,0,0},\gamma _{1,1,0},\gamma _{2,0,0}$.
	\end{enumerate}
\end{lemma}

Notice that, since we are assuming that the origin is an invisible cusp-fold singularity of $Z_0$ with degree $2$ and index $L_0\neq 0$, it follows that $g_3(0,0)|_{\dg=0}=L_0\neq0$. Defining  $$L=g_3(0,0),$$ it follows from continuity, that $L\neq 0$ for small values of $\dg$.

\begin{prop}\label{prop1}
	Let $\overline{G_2}(x,c,\bg_{0,0,0})=\widetilde{G_2}(x,c,\bg_{0,0,0})/x$, where $\widetilde{G_2}$ are given by Lemma \ref{zerofunclem}. Since $L\neq 0$, the function $\overline{G_2}(x,c,\bg_{0,0,0})=0$ has
	\begin{enumerate}[i)]
		\item two branches of zeroes $(x_+(c,\bg_{0,0,0}),c,\bg_{0,0,0})$ and $(x_-(c,\bg_{0,0,0}),c,\bg_{0,0,0})$ when $\dfrac{1}{L\ag_{0,0,0}}(\bg_{0,0,0}-\ag_{0,0,0}c)>0$;
		\item a unique branch of zeroes $(0,c,\bg_{0,0,0})$ when $\bg_{0,0,0}=\ag_{0,0,0}c$;
		\item no zeroes when $\dfrac{1}{L\ag_{0,0,0}}(\bg_{0,0,0}-\ag_{0,0,0}c)<0$.
	\end{enumerate}
	Moreover, the zeroes $x_{\pm}$ in the item $(i)$ are given by:
	$$x_{\pm}(c,\bg_{0,0,0})= \pm2\sqrt{\dfrac{1}{L\ag_{0,0,0}}(\bg_{0,0,0}-\ag_{0,0,0}c) }(1+\er(|(c,\bg_{0,0,0})|))+\er(\bg_{0,0,0}-\ag_{0,0,0}c).$$
\end{prop}
\begin{proof}
	Notice that $\overline{G_2}$ satisfies that $\overline{G_2}(0,0,0)=\partial_x \overline{G_2}(0,0,0)=0$ and $$\partial^2_x \overline{G_2}(0,0,0)=L\neq 0.$$
		From Malgrange Preparation Theorem, there exist smooth functions $K(x,c,\bg_{0,0,0})$, $A_0(c,\bg_{0,0,0})$ and $A_1(c,\bg_{0,0,0})$ such that
	$$\overline{G_2}(x,c,\bg_{0,0,0})=K(x,c,\bg_{0,0,0})(A_0(c,\bg_{0,0,0})+ A_1(c,\bg_{0,0,0})x+x^2),$$
	and $K(0,0,0)=L$. Thus, \begin{equation}\label{g2malg}
		\overline{G_2}(x,c,\bg_{0,0,0})=LA_0(c,\bg_{0,0,0})+(LA_1(c,\bg_{0,0,0})+A_0(c,\bg_{0,0,0})\partial_xK(0,c,\bg_{0,0,0}))x+\er(x^2),
	\end{equation} which means that 
	$$A_0(c,\bg_{0,0,0})=\dfrac{g_1(c,\bg_{0,0,0})}{L}= -\dfrac{4}{L\ag_{0,0,0}}(\bg_{0,0,0}-\ag_{0,0,0}c)+\er(|(c,\bg_{0,0,0})|^2),$$
	and 
	$$A_1(c,\bg_{0,0,0})=\dfrac{g_2(c,\bg_{0,0,0})}{L}-A_0(c,\bg_{0,0,0})\partial_xK(0,c,\bg_{0,0,0})= \er(c,\bg_{0,0,0}).$$
	
	Also, we notice that $A_0(c,\ag_{0,0,0}c)=A_1(c,\ag_{0,0,0}c)=0$. Thus, we can see that
	
	$$A_0(c,\bg_{0,0,0})= -\dfrac{4}{L\ag_{0,0,0}}(\bg_{0,0,0}-\ag_{0,0,0}c)+(\bg_{0,0,0}-\ag_{0,0,0}c)\er(|(c,\bg_{0,0,0})|),$$
	and 
	$$A_1(c,\bg_{0,0,0})=\er(\bg_{0,0,0}-\ag_{0,0,0}c).$$

	 Since $K$ is nonzero for $(x,c,\bg_{0,0,0})$ sufficiently small, it follows that the zeroes of $\overline{G_2}$ are given by the zeroes of 
	$$P_{c,\bg_{0,0,0}}(x)=A_0(c,\bg_{0,0,0})+ A_1(c,\bg_{0,0,0})x+x^2.$$
	The discriminant of $P$, $$\Delta(c,\bg_{0,0,0})= A_1(c,\bg_{0,0,0})^2-4A_0(c,\bg_{0,0,0})$$ is given by
	$$\Delta(c,\bg_{0,0,0})=\dfrac{16}{L\ag_{0,0,0}}(\bg_{0,0,0}-\ag_{0,0,0}c)(1+\er(|(c,\bg_{0,0,0})|)).$$
	From the Implicit Function Theorem, there exists a unique function $\bg_{0,0,0}^*(c)$ such that $\Delta(c,\bg_{0,0,0})=0$ if and only if $\bg_{0,0,0}=\bg_{0,0,0}^*(c)$. Nevertheless, $\Delta(c,\ag_{0,0,0}c)=0$, which means that $\Delta(c,\bg_{0,0,0})=0$ if and only if $\bg_{0,0,0}=c\ag_{0,0,0}$. The result follows directly.
	
	When $\dfrac{1}{L\ag_{0,0,0}}(\bg_{0,0,0}-\ag_{0,0,0}c)>0$, the zeroes of $P_{c,\bg_{0,0,0}}$ are given by
	$$x_{\pm}(c,\bg_{0,0,0})=\dfrac{-A_1(c,\bg_{0,0,0})\pm \sqrt{\Delta(c,\bg_{0,0,0})}}{2},$$
which express as 
 $$x_{\pm}(c,\bg_{0,0,0})= \pm2\sqrt{\dfrac{1}{L\ag_{0,0,0}}(\bg_{0,0,0}-\ag_{0,0,0}c) }(1+\er(|(c,\bg_{0,0,0})|))+\er(\bg_{0,0,0}-\ag_{0,0,0}c)$$
 	
	Also, we notice that when $\bg_{0,0,0}=\ag_{0,0,0}c$, 
	$$x_+(c,\ag_{0,0,0}c)=x_-(c,\ag_{0,0,0}c)=0.$$

\end{proof}

\begin{remark}
	From construction, it follows that $\Pi_Y(x_+,\textbf{y}(x_+,c),0)=(x_-,\textbf{y}(x_-,c),0)$.
\end{remark}

Now, we study the location of the zeroes $p_{\pm}(c,\bg_{0,0,0})=\left(x_{\pm}(c,\bg_{0,0,0}),\textbf{y}(x_{\pm}(c,\bg_{0,0,0}),c,\bg_{0,0,0})\right)$ when they exist.

\begin{prop}\label{propbifcurve}
	If $\dfrac{1}{L\ag_{0,0,0}}(\bg_{0,0,0}-\ag_{0,0,0}c)>0$, then the following statements hold:
\begin{enumerate}
	\item There exists a unique germ of function $\bg_{0,0,0}^*:(\rn{},0)\rightarrow(\rn{},0)$ given by
	$$\bg_{0,0,0}^*(c)=\ag_{0,0,0}c+\dfrac{3L\ag_{0,0,0}}{2}c^2+\er(c^3),$$
	such that $p_{-}(c,\bg_{0,0,0})$ is in the singular set of $X$ if and only if $\bg_{0,0,0}=\bg_{0,0,0}^*(c)$. Moreover, it is a visible fold-regular point  of $X$ when $c>0$ and an invisible fold-regular point of $X$ when $c<0$.
	\item The point $p_{+}(c,\bg_{0,0,0})$ is not in the singular set of $X$, for $c,\bg_{0,0,0}$ sufficiently small.
\end{enumerate}	


\end{prop}
\begin{proof}

If $\dfrac{1}{L\ag_{0,0,0}}(\bg_{0,0,0}-\ag_{0,0,0}c)>0$, then $p_{\pm}(c,\bg_{0,0,0})$ is a visible fold-regular point of $X$ if and only if $\textbf{y}(x_{\pm}(c,\bg_{0,0,0}),c,\bg_{0,0,0})=0$ and $x_{\pm}(c,\bg_{0,0,0})<-c$, where \textbf{y} is given by \eqref{eqy}. In this case, $p_{\pm}(c,\bg_{0,0,0})$ will correspond to a polycycle through a fold-regular point. In what follows, we look for the curve in the $(c,\bg_{0,0,0})-$parameter space where $Z$  presents such polycycle.

Notice that, if $(x,c,\bg_{0,0,0})$ satisfies the system
$$
\left\{\begin{array}{lcl}
	\overline{G_2}(x,c,\bg_{0,0,0}) &=&0,\\
	\overline{\textbf{y}}(x,c,\bg_{0,0,0}) &=&0,
\end{array}\right.
$$
where $\overline{\textbf{y}}(x,c,\bg_{0,0,0})=\textbf{y}(x,c,\bg_{0,0,0})/x$, then the point $(x,0)$ lies in the tangential line of the vector field $X$.

From \eqref{g2malg}, we have that
$$\overline{G_2}(x,c,\bg_{0,0,0})=4c-\dfrac{4}{\ag_{0,0,0}}\bg_{0,0,0}+\er(|(c,\bg_{0,0,0})|^2)+\er(x),$$
and we have that $$\overline{\textbf{y}}(x,c,\bg_{0,0,0})= y_1(c,\bg_{0,0,0})+ y_2(c,\bg_{0,0,0})x+\er(x^2),$$
where $y_1(c,\bg_{0,0,0})=c+\er(|(c,\bg_{0,0,0})|^2)$ and $y_2(c,\bg_{0,0,0})=\frac{1}{3}+\er(c,\bg_{0,0,0})$.
Therefore, considering $F(x,c,\bg_{0,0,0})=(\overline{G_2}(x,c,\bg_{0,0,0}),\overline{\textbf{y}}(x,c,\bg_{0,0,0}))$, we have that $F(0,0,0)=(0,0)$ and 
$$\dfrac{\partial F}{\partial(x,\bg_{0,0,0})}(0,0,0)=\dfrac{4}{3\ag_{0,0,0}}\neq0.$$
Thus, from the Implicit Function Theorem we obtain a curve $(x^*(c),\bg_{0,0,0}^*(c))$ for which $F(x,c,\bg_{0,0,0})=(0,0)$ if and only if $x=x^*(c)$ and $\bg=\bg_{0,0,0}^*(c)$.

Notice that $\bg_{0,0,0}^*(0)=0$, $x^*(0)=0$ and from the expression of $F$, it follows that
$(\bg_{0,0,0}^{*}) '(0)=\ag_{0,0,0}$ and $(x^{*})'(0)=-3,$ which means that $$x^*(c)=-3c+\er(c^2) \textrm{ and }\bg_{0,0,0}^{*}(c)=\ag_{0,0,0}c+Bc^2+\er(c^3),$$
where $B$ is a constant to be computed.

Since $(x_{\pm}(c,\bg_{0,0,0}),c,\bg_{0,0,0})$ (given by Proposition \ref{prop1}) or $(0,c,\ag_{0,0,0}c)$ are the unique zeroes of $\overline{G_2}$, and $x^*(c)\neq 0$ for $c\neq0$, it follows that $\bg_{0,0,0}^*(c)\neq \ag_{0,0,0}c$,  $\dfrac{1}{L\ag_{0,0,0}}\left(\bg_{0,0,0}^*(c)-\ag_{0,0,0}c\right)>0$ and either $$x^*(c)=x_{+}(c,\bg_{0,0,0}^*(c))\textrm{ or }x^*(c)=x_{-}(c,\bg_{0,0,0}^*(c)).$$


From the expression of $x^*(c)$ we have that $x^*(c)<-c$, for $c$ sufficiently small, if and only if $c>0$. It means that $(x^*(c),0)$ is a visible fold-regular point when $c>0$ and it is an invisible fold-regular point when $c<0$.

Now, assume that $c>0$ and notice that 
$$A_0(c,\bg_{0,0,0}^*(c))=\dfrac{-4B}{L\ag_{0,0,0}}c^2+\er(c^3)$$
and 
$$A_1(c,\bg_{0,0,0}^*(c))=\dfrac{g_2(c,\bg_{0,0,0}^*(c))}{L}-A_0(c,\bg_{0,0,0}^*(c))\partial_xK(0,c,\bg_{0,0,0}^*(c))=\er(c^2).$$
Since $\Delta(c,\bg_{0,0,0}^*(c))=A_1(c,\bg_{0,0,0}^*(c))^2-4A_0(c,\bg_{0,0,0}^*(c))$, it follows that
$$\Delta(c,\bg_{0,0,0}^*(c))=\dfrac{16B}{L\ag_{0,0,0}}c^2+\er(c^3),$$
hence
$$x_{\pm}(c,\bg_{0,0,0}^*(c))=\pm\sqrt{\dfrac{4B}{L\ag_{0,0,0}}}c+\er(c^{3/2}).$$

Since, either $x^*(c)=x_{+}(c,\bg_{0,0,0}^*(c))$ or $x^*(c)=x_{-}(c,\bg_{0,0,0}^*(c))$, and $x^*(c)=-3c+\er(c^2)$, it follows that $x^*(c)=x_{-}(c,\bg_{0,0,0}^*(c))$ and $$B=\dfrac{9L\ag_{0,0,0}}{4}.$$

\end{proof}

\begin{prop}\label{propzeroesloc}
	If $\dfrac{1}{L\ag_{0,0,0}}(\bg_{0,0,0}-\ag_{0,0,0}c)>0$, $|\bg_{0,0,0}-\ag_{0,0,0}c|<|\bg_{0,0,0}^*(c)-\ag_{0,0,0}c|$, and $c>0$ sufficiently small, then $p_-(c,\bg_{0,0,0})\in\s^c$ if and only if $\cg_{1,0,0}>0$.
\end{prop}
\begin{proof}
	Recall that $p_-(c,\bg_{0,0,0})=\left(x_{-}(c,\bg_{0,0,0}),\textbf{y}(x_{-}(c,\bg_{0,0,0}),c,\bg_{0,0,0})\right)$. From Proposition \ref{prop1}, we have that
	$$x_{-}(c,\bg_{0,0,0})= -2\sqrt{\dfrac{1}{L\ag_{0,0,0}}(\bg_{0,0,0}-\ag_{0,0,0}c) }(1+\er(|(c,\bg_{0,0,0})|))+\er(\bg_{0,0,0}-\ag_{0,0,0}c).$$
	Now, 
	$$\textbf{y}(x,c,\bg_{0,0,0})=x\left(c+\er(|(c,\bg_{0,0,0})|^2)+\left(\dfrac{1}{3}+\er(|(c,\bg_{0,0,0})|)\right)x+\er(x^2)\right).$$
	
	Since $|\bg_{0,0,0}-\ag_{0,0,0}c|<|\bg_{0,0,0}^*(c)-\ag_{0,0,0}c|$, we have that  $|\bg_{0,0,0}-\ag_{0,0,0}c|<\dfrac{3|L\ag_{0,0,0}|}{2}c^2+\er(c^3)$, $\bg_{0,0,0}=\er(c)$ and $\bg_{0,0,0}-\ag_{0,0,0}c=\er(c^2).$ Hence we have that
	$$\textbf{y}(x_-(c,\bg_{0,0,0}),c,\bg_{0,0,0})= x_{-}(c,\bg_{0,0,0})\left(c+ \dfrac{1}{3}x_-(c,\bg_{0,0,0})+\er(c^2)\right).$$
	Thus, 
	$$Xf(p_-(c,\bg_{0,0,0}))=x_{-}(c,\bg_{0,0,0})\left(c+ \dfrac{1}{3}x_-(c,\bg_{0,0,0})+\er(c^2)\right),$$
	and since $Yf(x,y)=\cg_{1,0,0}x+\cg_{0,1,0}y+\er(|(x,y)|^2)$, it follows that
	$$Yf(p_-(c,\bg_{0,0,0}))=x_{-}(c,\bg_{0,0,0})\left(\gamma_{1,0,0}+\gamma_{0,1,0}\left(c+ \dfrac{1}{3}x_-(c,\bg_{0,0,0})\right)+\er(c)\right).$$
		It means that 
	$$Xf(p_-(c,\bg_{0,0,0}))Yf(p_-(c,\bg_{0,0,0}))=(x_-(c,\bg_{0,0,0}))^2\left(\cg_{1,0,0}\left(c+\dfrac{1}{3}x_-(c,\bg_{0,0,0})\right) +\er(c^2)\right).$$
		Since $x_-(c,\ag_{0,0,0}c)=0$ and $x_-(c,\bg_{0,0,0}^*(c))=-3c+\er(c^2)$, it follows that, for $\ag_{0,0,0}c<\bg_{0,0,0}<\bg_{0,0,0}^*(c)$, we have that $-3c+\er(c^2)<x_-(c,\bg_{0,0,0})<0,$ for $c>0$ sufficiently small. Hence, $$c+\dfrac{1}{3}x_-(c,\bg_{0,0,0})>0.$$
		Thus, it follows that $Xf(p_-(c,\bg_{0,0,0}))Yf(p_-(c,\bg_{0,0,0}))>0$ if and only if $\cg_{1,0,0}>0$.

\end{proof}

\begin{prop}
	If $\dfrac{1}{L\ag_{0,0,0}}(\bg_{0,0,0}-\ag_{0,0,0}c)>0$, $|\bg_{0,0,0}-\ag_{0,0,0}c|\leq|\bg_{0,0,0}^*(c)-\ag_{0,0,0}c|$, and $c>0$ sufficiently small, then $p_+(c,\bg_{0,0,0})\in\s^c$ if and only if $\cg_{1,0,0}>0$.
\end{prop}	
\begin{proof}
	Recall that $p_+(c,\bg_{0,0,0})=\left(x_{+}(c,\bg_{0,0,0}),\textbf{y}(x_{+}(c,\bg_{0,0,0}),c,\bg_{0,0,0})\right)$. From Proposition \ref{prop1}, we have that
	$$x_{+}(c,\bg_{0,0,0})= 2\sqrt{\dfrac{1}{L\ag_{0,0,0}}(\bg_{0,0,0}-\ag_{0,0,0}c) }(1+\er(|(c,\bg_{0,0,0})|))+\er(\bg_{0,0,0}-\ag_{0,0,0}c).$$
	Now, 
	$$\textbf{y}(x,c,\bg_{0,0,0})=x\left(c+\er(|(c,\bg_{0,0,0})|^2)+\left(\dfrac{1}{3}+\er(|(c,\bg_{0,0,0})|)\right)x+\er(x^2)\right).$$
	Hence
$$\textbf{y}(x_+(c,\bg_{0,0,0}),c,\bg_{0,0,0})=x_{+}(c,\bg_{0,0,0})\left(c+ \dfrac{1}{3}x_{+}(c,\bg_{0,0,0})+ g(c,\bg_{0,0,0})\right),$$
where
 $$g(c,\bg_{0,0,0})=\sqrt{\dfrac{1}{L\ag_{0,0,0}}(\bg_{0,0,0}-\ag_{0,0,0}c) }\er(|(c,\bg_{0,0,0})|)+ \er(\bg_{0,0,0}-\ag_{0,0,0}c)+ \er(|(c,\bg_{0,0,0})|^2).$$

	Thus, 
$$Xf(p_+(c,\bg_{0,0,0}))=x_{+}(c,\bg_{0,0,0})\left(c+ \dfrac{1}{3}x_+(c,\bg_{0,0,0})+g(c,\bg_{0,0,0})\right),$$
and since $Yf(x,y)=\cg_{1,0,0}x+\cg_{0,1,0}y+\er(|(x,y)|^2)$, it follows that
$$Yf(p_+(c,\bg_{0,0,0}))=x_{+}(c,\bg_{0,0,0})\left(\gamma_{1,0,0}+\gamma_{0,1,0}c+\er(x_+(c,\bg_{0,0,0}))\right).$$
Hence, 
$$Xf(p_+(c,\bg_{0,0,0}))Yf(p_+(c,\bg_{0,0,0}))=(x_{+}(c,\bg_{0,0,0}))^2\left(\cg_{1,0,0}\left(c+\dfrac{1}{3}x_{+}(c,\bg_{0,0,0})\right)+\er(g(c,\bg_{0,0,0})) \right).$$
Notice that $\bg_{0,0,0}=\er(c)$ and $\bg_{0,0,0}-\ag_{0,0,0}c=\er(c^2)$. It follows that $g(c,\bg_{0,0,0})=\er(c^2)$ and $x_+(c,\bg_{0,0,0})=\er(c)$. Since $x_+(c,\bg_{0,0,0})>0$, it follows that  $Xf(p_+(c,\bg_{0,0,0}))Yf(p_+(c,\bg_{0,0,0}))>0$ for every $(c,\bg_{0,0,0})$ such that $|\bg_{0,0,0}-\ag_{0,0,0}c|\leq|\bg_{0,0,0}^*(c)-\ag_{0,0,0}c|$ and $c\geq  0$.
\end{proof}

Finally, from Propositions \ref{propbifcurve} and \ref{propzeroesloc}, we  conclude that, if $\dfrac{1}{L\ag_{0,0,0}}(\bg_{0,0,0}-\ag_{0,0,0}c)>0$, $|\bg_{0,0,0}-\ag_{0,0,0}c|\leq|\bg_{0,0,0}^*(c)-\ag_{0,0,0}c|$, $\cg_{1,0,0}>0$ and $c>0$ sufficiently small, then 
	\begin{enumerate}
		\item $Z$ has a unique (one-loop) CLC for every $(c,\bg_{0,0,0})$ such that $|\bg_{0,0,0}-\ag_{0,0,0}c|<|\bg_{0,0,0}^*(c)-\ag_{0,0,0}c|.$
		\item $Z$ has a unique (one-loop) polycycle passing through a fold-regular point for every $(c,\bg_{0,0,0})$ such that $\bg_{0,0,0}=\bg_{0,0,0}^*(c).$
	\end{enumerate}

It completes the proof of Theorem \ref{Teo-CLC-Policiclo}.

\begin{remark}
	For $\cg_{1,0,0}>0$, the fixed points $p_-(c,\bg_{0,0,0})$ and $p_+(c,\bg_{0,0,0})$ of the Poincaré map occur on the curve $y=\textbf{y}(x,c,\bg_{0,0,0})$ which is locally a parabola $y=cx+\dfrac{1}{3}x^2+h.o.t.$ tangent to the line $y=c x$. When $c>0$ is sufficiently small, then the line $y=cx$ is contained in the crossing region, which illustrates the fact that $p_{\pm}\in\s^c$ when  $|\bg_{0,0,0}-\ag_{0,0,0}c|<|\bg_{0,0,0}^*(c)-\ag_{0,0,0}c|$. The fixed points $p_{\pm}$ birth at the origin when $\bg_{0,0,0}=\ag_{0,0,0}c$ and bifurcates, one at each branch of the parabola, when $\dfrac{1}{L\ag_{0,0,0}}(\bg_{0,0,0}-\ag_{0,0,0}c)>0$. In the positive branch $p_+$ will always remain on the crossing region. Nevertheless, the negative branch of the parabola presents an intersection with the tangency line $S_X$ of $X$, thus $p_-$ births at the origin when $\bg_{0,0,0}=\ag_{0,0,0}c$ and remains on the crossing region until the curve $\bg_{0,0,0}=\bg_{0,0,0}^*(c)$, where  $p_-$ dies at the intersecting point of $S_X$ and outside such described region in the parameter space, $p_-$ belongs to the escaping region where it looses dynamical meaning. So,  in the region delimited by the curves $\bg_{0,0,0}=\ag_{0,0,0}c$ and $\bg_{0,0,0}=\bg_{0,0,0}^*(c)$, a CLC bifurcates from the origin, which is an invisible fold-fold point, it grows when the parameters vary in the region until it dies on a polycycle through a fold-regular singularity when the parameters reach the curve $\bg_{0,0,0}=\bg_{0,0,0}^*(c)$.
	
	The situations are described in the Figure \ref{Fig-Cases}.
\end{remark}

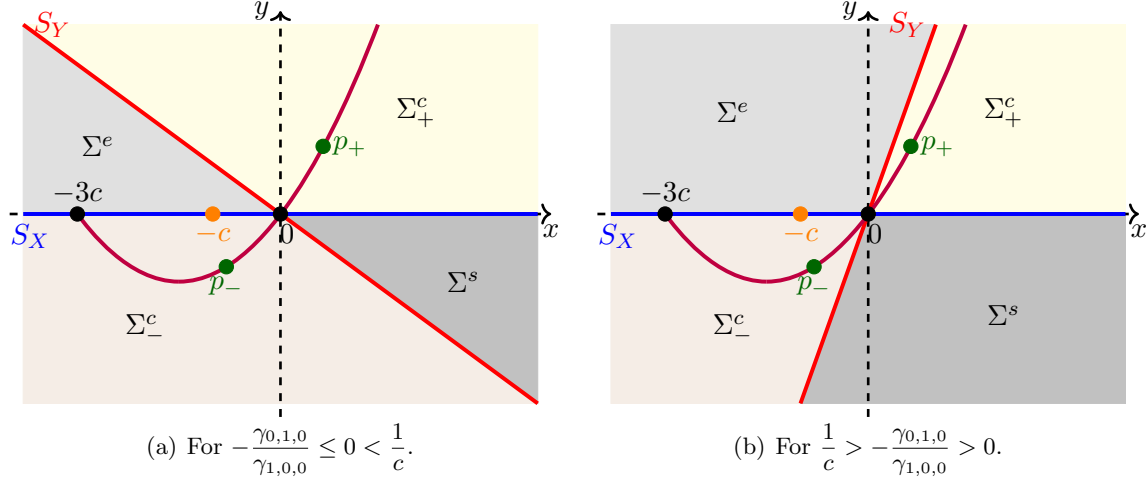
\begin{figure}[h]
\begin{center}
  \subfigure[For $-\dfrac{\cg_{0,1,0}}{\cg_{1,0,0}}\leq 0<\dfrac{1}{c}$.]{
\begin{tikzpicture}[scale = 0.9]
\filldraw[opacity=0.4,black!30!white] (-3.8,0) to (-3.8,2.8) to (0,0);   
  \filldraw[opacity=0.4,black!60!white] (3.8,0) to (3.8,-2.8) to (0,0); 
\filldraw[opacity=0.1,yellow] (3.8,0) to (3.8,2.8) to (-3.8,2.8) to (0,0);   
   \filldraw[opacity=0.1,orange!60!black] (-3.8,0) to (-3.8,-2.8) to (3.8,-2.8) to (0,0); 
   
  \draw[line width=1pt,dashed,->] (-4,0) to (4,0)node[below]{$x$};
  \draw[line width=1pt,dashed,->] (0,-3) to (0,3)node[left]{$y$};

  \draw[line width=1.5pt, blue] (-3.8,0) to (3.8,0);
  \draw[blue] (-3.7,0)node[below]{$S_X$};
  \draw[line width=1.5pt, red] (3.8,-2.8) to (-3.8,2.8);
  \draw[red] (-3.8,2.8)node[right]{$S_Y$};
  \draw[] (2,1.5)node[]{$\Sigma^{c}_+$};
  \draw[] (-2,-1.7)node[]{$\Sigma^{c}_-$};
  \draw[] (2.7,-1)node[]{$\Sigma^{s}$};
  \draw[] (-2.7,1)node[]{$\Sigma^{e}$};

      \draw[line width=1.5pt,purple] (-1.5,-1) parabola (-3,0); 
      \draw[line width=1.5pt,purple] (-1.5,-1) parabola (1.44,2.8);
              
       \filldraw[black] (0,0) circle (3pt);
       \draw[] (0.1,0)node[below]{$0$};
       
       \filldraw[orange] (-1,0)node[below]{$-c$} circle (3pt);
       \filldraw[black] (-3,0)node[above]{$-3c$} circle (3pt);
       
       \filldraw[black!60!green] (-0.8,-0.78)node[below]{$p_-$} circle (3pt);
        \filldraw[black!60!green] (0.63,1)node[right]{$p_+$} circle (3pt);

  \end{tikzpicture}}
  \subfigure[For $\dfrac{1}{c}>-\dfrac{\cg_{0,1,0}}{\cg_{1,0,0}}>0$.]{
\begin{tikzpicture}[scale = 0.9]
\filldraw[opacity=0.4,black!30!white] (-3.8,0) to (-3.8,2.8) to (1,2.8) to (0,0);   
  \filldraw[opacity=0.4,black!60!white] (3.8,0) to (3.8,-2.8) to (-1,-2.8) to (0,0); 
\filldraw[opacity=0.1,yellow] (3.8,0) to (3.8,2.8) to (1,2.8) to (0,0);   
   \filldraw[opacity=0.1,orange!60!black] (-3.8,0) to (-3.8,-2.8) to (-1,-2.8) to (0,0);
   
  \draw[line width=1pt,dashed,->] (-4,0) to (4,0)node[below]{$x$};
  \draw[line width=1pt,dashed,->] (0,-3) to (0,3)node[left]{$y$};
  
  \draw[line width=1.5pt, blue] (-3.8,0) to (3.8,0);
  \draw[blue] (-3.7,0)node[below]{$S_X$};
  \draw[line width=1.5pt, red] (-1,-2.8) to (1,2.8);
  \draw[red] (1,2.8)node[left]{$S_Y$};
  \draw[] (2,1.5)node[]{$\Sigma^{c}_+$};
  \draw[] (-2,-1.7)node[]{$\Sigma^{c}_-$};
  \draw[] (2,-1.5)node[]{$\Sigma^{s}$};
  \draw[] (-2,1.5)node[]{$\Sigma^{e}$};

      \draw[line width=1.5pt,purple] (-1.5,-1) parabola (-3,0); 
      \draw[line width=1.5pt,purple] (-1.5,-1) parabola (1.44,2.8);
              
       \filldraw[black] (0,0) circle (3pt);
       \draw[] (0.1,0)node[below]{$0$};
       
       \filldraw[orange] (-1,0)node[below]{$-c$} circle (3pt);
       \filldraw[black] (-3,0)node[above]{$-3c$} circle (3pt);
       
       \filldraw[black!60!green] (-0.8,-0.78)node[below]{$p_-$} circle (3pt);
        \filldraw[black!60!green] (0.63,1)node[right]{$p_+$} circle (3pt);

  \end{tikzpicture}}
  \caption{Possible scenarios on $\Sigma$ for $\cg_{1,0,0}>0$ and $c>0$. Disregarding higher order terms, the tangency line $S_Y$ is $\cg_{1,0,0}x+\cg_{0,1,0}y=0$ and the curve where the fixed points $p_-$ and $p_+$ live is the parabola $y=cx+\frac{1}{3}x^2$ (purple line). So, $p_-=(-3c,0)$ and $p_+\in\Sigma^c_+$ when $\bg_{0,0,0}=\bg_{0,0,0}^*(c)=\ag_{0,0,0}c+\frac{3}{2}L\ag_{0,0,0}c^2$, $p_-=p_+=(0,0)$ when $\bg_{0,0,0}=\ag_{0,0,0}c$, and $p_{\pm}\in\Sigma^c_{\pm}$ whenever $0<\bg_{0,0,0}-\ag_{0,0,0}c<\frac{3}{2}L\ag_{0,0,0}c^2$ or $\frac{3}{2}L\ag_{0,0,0}c^2<\bg_{0,0,0}-\ag_{0,0,0}c<0$. Note that the cusp singularity (orange dot) is always at the boundary of the escaping region. However, all results presented here are also valid for the case where the cusp singularity is at the boundary of the sliding region and, to obtain this, simply take the vector fields with opposite directions, $-X$ and $-Y$. }\label{Fig-Cases}
\end{center}
\end{figure}


\section*{Acknowledgements}
Rony Cristiano acknowledges CNPq/Brazil for funding his work under Grant PQ-310281/2023-7. Oscar A. R. Cespedes acknowledges ... Ot\'avio M. L. Gomide acknowledges ...

\bibliographystyle{siam}
\bibliography{References-Polycycle}

\appendix

\section{Algorithm to obtain coefficients of the Half-Return Map $\Pi_Y$ given by Lemma \ref{foldreturnlem}}\label{coeffpiY}

The following algorithm was used to compute the expansion of $\Pi_Y$ through Mathematica.

\begin{lstlisting}
(* General 4th-degree polynomial definition *)
Q[p_, x_, y_, z_] := Sum[Sum[Sum[Subscript[p, i, j, l] x^i y^j z^l, 
  {i, 0, 4}], {j, 0, 4}], {l, 0, 4}];

(* Constrained polynomial with 5th-derivative condition *)
P[p_, x_, y_, z_] := Q[p, x, y, z] /. 
   Solve[Variables[D[Q[p, x, y, z] /. {x -> t, y -> t, z -> t}, {t, 5}] /. t -> 1] == 0][[1]];

(* ====================== *)
(* Defining Vector Fields *)
(* ====================== *)

F2[x_, y_, z_] := {P[\[Alpha], x, y, z], P[\[Beta], x, y, z], P[\[Gamma], x, y, z]};

Subscript[\[Gamma], 0, 0, 0] = 0;

(* ========================================= *)
(* Taylor expansion up to 3rd order of field solution *)
(* with initial condition (x, y, 0) *)
(* ========================================= *)

\[Phi][t_, x_, y_] := {P[p, t, x, y], P[r, t, x, y], P[s, t, x, y]};

C0 = FullSimplify[
   Solve[CoefficientList[CoefficientList[(\[Phi][0, x, y] - {x, y, 0}), x], y] == 0][[1]]];

C1 = FullSimplify[
  Solve[((D[\[Phi][t, x, y], t] - 
        F2[P[p, t, x, y], P[r, t, x, y], P[s, t, x, y]]) /. {t -> 0, x -> 0, y -> 0} /. C0) == 0, 
    {Subscript[s, 1, 0, 0], Subscript[p, 1, 0, 0], Subscript[r, 1, 0, 0]}][[1]]];

C2 = FullSimplify[
   Solve[(D[(D[\[Phi][t, x, y], t] - 
         F2[P[p, t, x, y], P[r, t, x, y], P[s, t, x, y]]) /. C0 /. C1, {{t, x, y}}] /. 
       {t -> 0, x -> 0, y -> 0}) == 0][[1]]];

C3 = FullSimplify[
   Solve[Union @@ Union @@ 
      Factor[D[(D[\[Phi][t, x, y], t] - 
            F2[P[p, t, x, y], P[r, t, x, y], P[s, t, x, y]]] /. C1 /. C2 /. C0, 
          {{x, y, t}, 2}] /. {x -> 0, y -> 0, t -> 0}] == 0, 
    {Subscript[p, 1, 0, 2], Subscript[p, 1, 1, 1], Subscript[p, 1, 2, 0],
     Subscript[p, 2, 0, 1], Subscript[p, 2, 1, 0], Subscript[p, 3, 0, 0],
     Subscript[r, 1, 0, 2], Subscript[r, 1, 1, 1], Subscript[r, 1, 2, 0],
     Subscript[r, 2, 0, 1], Subscript[r, 2, 1, 0], Subscript[r, 3, 0, 0],
     Subscript[s, 1, 0, 2], Subscript[s, 1, 1, 1], Subscript[s, 1, 2, 0],
     Subscript[s, 2, 0, 1], Subscript[s, 2, 1, 0], Subscript[s, 3, 0, 0]}][[1]]];

C4 = FullSimplify[
   Solve[Union @@ Union @@ 
      Factor[D[(D[\[Phi][t, x, y], t] - 
            F2[P[p, t, x, y], P[r, t, x, y], P[s, t, x, y]]] /. C1 /. C2 /. C0 /. C3, 
          {{x, y, t}, 3}] /. {x -> 0, y -> 0, t -> 0}] == 0, 
    {Subscript[p, 1, 0, 3], Subscript[p, 1, 1, 2], Subscript[p, 1, 2, 1],
     Subscript[p, 1, 3, 0], Subscript[p, 2, 0, 2], Subscript[p, 2, 1, 1],
     Subscript[p, 2, 2, 0], Subscript[p, 3, 0, 1], Subscript[p, 3, 1, 0],
     Subscript[p, 4, 0, 0], Subscript[r, 1, 0, 3], Subscript[r, 1, 1, 2],
     Subscript[r, 1, 2, 1], Subscript[r, 1, 3, 0], Subscript[r, 2, 0, 2],
     Subscript[r, 2, 1, 1], Subscript[r, 2, 2, 0], Subscript[r, 3, 0, 1],
     Subscript[r, 3, 1, 0], Subscript[r, 4, 0, 0], Subscript[s, 1, 0, 3],
     Subscript[s, 1, 1, 2], Subscript[s, 1, 2, 1], Subscript[s, 1, 3, 0],
     Subscript[s, 2, 0, 2], Subscript[s, 2, 1, 1], Subscript[s, 2, 2, 0],
     Subscript[s, 3, 0, 1], Subscript[s, 3, 1, 0], Subscript[s, 4, 0, 0]}][[1]]];

\[Rho][t_, x_, y_] := \[Phi][t, x, y] /. C1 /. C2 /. C0 /. C3 /. C4;

(* ============================================== *)
(* Taylor expansion of return time *)
(* for condition (x, y, 0) *)
(* ============================================== *)

K1 = FullSimplify[
   Solve[(D[Expand[(\[Rho][t, x, y][[3]]/t)] /. t -> t[x, y], {{x, y}}] /. 
       {x -> 0, y -> 0} /. t[0, 0] -> 0) == 0, 
    {Derivative[0, 1][t][0, 0], Derivative[1, 0][t][0, 0]}][[1]]];

K2 = FullSimplify[
   Solve[(D[Expand[(\[Rho][t, x, y][[3]]/t)] /. t -> t[x, y], {{x, y}, 2}] /. 
       {x -> 0, y -> 0} /. t[0, 0] -> 0 /. K1) == 0, 
    {Derivative[0, 2][t][0, 0], Derivative[1, 1][t][0, 0], Derivative[2, 0][t][0, 0]}][[1]]];

K3 = FullSimplify[
   Solve[(D[Expand[(\[Rho][t, x, y][[3]]/t)] /. t -> t[x, y], {{x, y}, 3}] /. 
       {x -> 0, y -> 0} /. t[0, 0] -> 0 /. K1 /. K2) == 0, 
    {Derivative[0, 3][t][0, 0], Derivative[1, 2][t][0, 0], 
     Derivative[2, 1][t][0, 0], Derivative[3, 0][t][0, 0]}][[1]]];

(* ============================================== *)
(* 3rd-order expansion of return map *)
(* ============================================== *)

{Py1[x_, y_], Py2[x_, y_]} = 
  Expand[Normal[
     Series[{\[Rho][t, x, y][[1]], \[Rho][t, x, y][[2]]} /. t -> t[x, y], 
      {x, 0, 3}, {y, 0, 3}]]] /. 
    {x^4 -> 0, x^5 -> 0, x^6 -> 0, x^3 y -> 0, x^4 y -> 0, x^5 y -> 0, 
     x^2 y^2 -> 0, x^3 y^2 -> 0, x^4 y^2 -> 0, x y^3 -> 0, x^2 y^3 -> 0, 
     x^3 y^3 -> 0, y^4 -> 0, x y^4 -> 0, x^2 y^4 -> 0, y^5 -> 0, x y^5 -> 0, 
     y^6 -> 0} /. t[0, 0] -> 0 /. K1 /. K2 /. K3;
\end{lstlisting}
\end{document}